\documentclass[11pt]{article}
\usepackage{amsmath,amsthm,amsfonts,amssymb,amscd, amsxtra, mathrsfs}
\usepackage{url}
\usepackage[margin=2.7 cm,nohead]{geometry}
\usepackage{color}
\usepackage{todonotes}
\usepackage{pdfsync}
\newtheorem{theorem}{Theorem}
\newtheorem{lemma}[theorem]{Lemma}
\newtheorem{definition}{Definition}
\newtheorem{corollary}[theorem]{Corollary}
\newtheorem{proposition}[theorem]{Proposition}

\newtheorem{remark}{Remark}

\usepackage{graphicx,float}

\usepackage{algorithm}
\usepackage{algpseudocode}

\newcommand{\myscale}{0.32}
\newcommand{\ds}{\displaystyle}
\newcommand{\R}{\mathbb{R}}
\newcommand{\tr}{\textrm{tr}}

\algnewcommand{\Input}[1]{%
  \State \textbf{Input:} {\raggedright #1}
  
}

\algnewcommand{\Initialize}[1]{%
  \State \textbf{Initialize:}
  \Statex \hspace*{\algorithmicindent}\parbox[t]{.8\linewidth}{\raggedright #1}
}

\algnewcommand{\Output}[1]{%
  \State \textbf{Output:} {\raggedright #1}
}

\begin{document}
\title{On the inexact scaled gradient projection method}
\author{
O. P. Ferreira\thanks{This author was  partially supported by FAPEG, CNPq Grants  302473/2017-3 and 408151/2016-1) and FAPEG/PRONEM- 201710267000532  (e-mail:{\tt
orizon@ufg.br}).}
\and
M. V. Lemes \thanks{ (e-mail:{\tt max@ufg.br}).}
\and
L. F. Prudente\thanks{This author was  partially supported by  FAPEG/PRONEM- 201710267000532 (e-mail:{\tt lfprudente@ufg.br}).  }
}

\maketitle

\maketitle
\begin{abstract}
	The purpose of this  paper is to present an inexact version of the scaled gradient projection method on a convex set, which  is inexact in two sense.  First, an inexact projection on the feasible set is computed, allowing for an appropriate relative error tolerance. Second, an  inexact non-monotone line search scheme is employed to compute  a step size which defines the next iteration. It is shown that the proposed   method has similar  asymptotic convergence properties and iteration-complexity bounds as the usual  scaled gradient projection method employing monotone line searches.
\end{abstract}



\section{Introduction}
This  paper is  devoted  to the study of  the {\it scaled gradient projection (SGP) method with non-monotone line search} to solve general  constrained convex optimization problems as follows
\begin{equation} \label{eq:OptP}
	\min \{ f(x) :~   x\in C\},
\end{equation}
where $C$ is a closed and convex subset of $\mathbb{R}^n$ and $f:\mathbb{R}^n \to \mathbb{R}$ is a continuously differentiable function. Denotes by $f^*:= \inf_{x\in C} f(x)$ the optimal value  of \eqref{eq:OptP} and by  $\Omega^*$  its  solution set, which we will assume to be nonempty unless the contrary is explicitly stated.  Problem~\eqref{eq:OptP} is a basic issue of constrained  optimization, which appears very often in various areas, including  finance,    machine learning, control theory, and signal processing, see for example \cite{Bottou_Curtis_Nocedal2018, Boyd_Ghaoui_Ferron1994, Figueiredo2007, Higham2002, Ma_Hu_Gao2015, Sra_Nowozin_Wright2012}.  Recent problems considered in most of these areas, the datasets are large or high-dimensional  and their solutions need to be approximated quickly with a reasonably accuracy. It is well known that SGP method with non-monotone line search is among those that are suitable for this task, as will be explained below.

The  gradient projection method   is what first comes to mind when we start from the ideas of the classic optimization methods in an attempt to deal with problem \eqref{eq:OptP}.  In fact, this  method is one of the oldest known optimization methods to solve \eqref{eq:OptP}, the study of its convergence properties goes back to the works of Goldstein \cite{Goldstein1964} and Levitin and Polyak \cite{Polyak_Levitin1966}.  After these works, several variants of it have appeared over the years, resulting in a vast literature on the subject, including  \cite{yunier_roman2010, Bertsekas1976, Bertsekas1999, Fan_Wang_Yan2019, Figueiredo2007, Gong2011,   Iusem2003, Patrascu_Necoara2018, Zhang_Wang_Yang2019}. Additional reference on this subject  can be found in the recent  review  \cite{bonettini2019recent} and  references therein. Among all the variants of the gradient projection method, the scaled  version has been especially considered due to the flexibility provided in  efficient  implementations of the method; see \cite{BirginMartinezRaydan2003,10.1093/imanum/drh020,Bonettini2016, BonettiniPrato2015, Bonettini2009}.  In addition, its simplicity and easy implementation has attracted the attention of the scientific community that works on optimization over the years.  This method usually uses only first-order derivatives, which makes it very stable from a numerical point of view and therefore quite suitable for solving large-scale optimization problems, see \cite{More1990, Nesterov_Nemirovski2013, Sra_Nowozin_Wright2012, tang_golbabaee_davies2017}. At each current iteration, SGP method  moves along the direction of the negative scaled gradient, and then projects the obtained point  onto the constraint set.  
The current iteration and such projection define a feasible descent direction and a line search in this direction is performed to define the next iteration. It is worth mentioning that the performance of SGM method is strongly related to each of the steps we have just  mentioned. In fact, the scale matrix and the step size towards the negative scaled gradient are freely selected in order to improve the performance  of SGM method but without increasing the cost of each iteration.  Strategies  for choosing both has its origin in the study of  gradient  method  for unconstrained  optimization,   papers addressing  this issues include  but not limited to \cite{BB1988, BonettiniPrato2015, DaiFletcher2006,DaiHage2006, Serafino2018, Friedlander1999, Dai2006,  DaiFletcher2005,  Polyak_Levitin1966}. More details  about about  selecting  step sizes and scale matrices  can be found in the recent  review  \cite{bonettini2019recent} and  references therein.

In this paper, we are particularly interested in the main stages that make up the SGP method, namely, in the projection calculation and in the line search employed.   It is well known that the mostly computational burden of each iteration of the SGP method is in the calculation of the projection.  Indeed, the   projection calculation requires, at each  iteration, the solution of a quadratic problem restricted to the feasible set,  which can lead to a substantial increase in the cost per iteration if the number of unknowns is large. For this reason, it may not be justified to carry out exact projections when the iterates are far from the solution of the problem. In order to reduce the computational effort spent on projections, inexact procedures that become more and more accurate when approaching the solution, have been proposed, resulting in more efficient methods;  see  for exemple   \cite{BirginMartinezRaydan2003, Bonettini2016,Golbabaee_Davies2018, Gonccalves2020, SalzoVilla2012, VillaSalzo2013, Rasch2020}.  On the other hand,  non-monotonous searches can improve the probability of finding an optimal global solution, in addition to potentially improving the speed of convergence of the method as a whole, see for example \cite{Dai2002, Panier1991, Toint1996}. The concept of non-monotone line search,  that we will use here as a synonym for  inexact line search,  have been proposed first in \cite{Grippo1986}, and  later a new non-monotone search was proposed in \cite{ZhangHager2004}.  After these papers  others  non-monotone searches appeared, see for example  \cite{Ahookhosh2012, MoLiuYan2007}.  In \cite{SachsSachs2011}, an interesting general framework for non-monotonous line research was proposed, and more recently modifications of it have been presented in \cite{GrapigliaSachs2017, GrapigliaSachs2020}.

The purpose of the present  paper is to present an inexact version of the SGP method, which  is inexact in two sense. First,  using  a version of  scheme introduced in \cite{BirginMartinezRaydan2003} and also a variation of the one appeared \cite[Example 1]{VillaSalzo2013},  the inexact projection  onto the feasible  set is computed  allowing an appropriate  relative error tolerance. Second,  using the inexact  conceptual scheme for the  line search  introduced  in  \cite{GrapigliaSachs2020, SachsSachs2011}, a step size is computed  to define the next iteration.   More specifically, initially we show that the  feasible inexact  projection of \cite{BirginMartinezRaydan2003} provides greater latitude than the projection of \cite[Example 1]{VillaSalzo2013}.  In the first  convergence result presented, we show that the SGP method using the projection proposed in \cite{BirginMartinezRaydan2003} preserves the same partial convergence result as the classic method, that is, we prove that every accumulation point  of the sequence generated by the SGP method is stationary for problem~\eqref{eq:OptP}. Then, considering the inexact projection of \cite[Example 1]{VillaSalzo2013}, and  under mild  assumptions,  we establish  full asymptotic convergence results  and  some complexity bounds. The  presented analysis of the method is done using the general  non-monotonous line search scheme  introduced in \cite{GrapigliaSachs2020}. In this way, the proposed method can be adapted to several line searches and, in particular, will allow obtaining several known versions of the SGP method as particular instances,  including  \cite{yunier_roman2010,BirginMartinezRaydan2003,Iusem2003,Xihong2018}. Except for the particular case when we assume that the SGP method employs  the non-monotonous line search introduced by \cite{Grippo1986}, all other  asymptotic convergence and complexity  results are obtained without any assumption of compactness of the sub-level sets of the objective function.  Finally, it is worth mentioning that the complexity results obtained  for the SGP method with a general non-monotone line search  are the same as in the classic case when the usual Armijo search is employed, namely,  the complexity bound  $\mathcal{O}(1/\sqrt{k})$ is unveil for finding $\epsilon$-stationary points for problem \eqref{eq:OptP} and, under convexity on $f$, the rate to find a $\epsilon$-optimal functional value is $\mathcal{O}(1/k)$.

In Section~\ref{Sec:Prel}, some notations and basic results used throughout the paper is presented. In particular, Section~\ref{Sec:SubInexProj} is devoted to recall the concept of relative feasible inexact projection and some  new properties about this concept are presented. Section~\ref{Sec:SGM}  describes the SGP method with a general non-monotone line search and some particular instances of it are presented.  Partial asymptotic convergence results  are presented in Section~\ref{Sec:PartialConvRes}. Section~\ref{Sec:FullConvRes}   presents  a full   convergence result  and iteration-complexity bounds. Some numerical experiments are provided in Section \ref{Sec:NumExp}. Finally, some concluding remarks are made in Section~\ref{Sec:Conclusions}.

\section{Preliminaries and basic results}  \label{Sec:Prel}

In this section, we introduce  some notation and results used throughout our presentation.  First we  consider the  index set  ${\mathbb{N}}:=\{0,1,2,\ldots\}$,  the usual inner  product  $\langle \cdot,\cdot \rangle$ in $\mathbb{R}^n$, and the associated Euclidean norm    $\|\cdot\|$.
Let  $f:\mathbb{R}^n \to \mathbb{R}$ be a differentiable function and $C \subseteq \mathbb{R}^n$. The  gradient $\nabla f$ of $f$ is said to be {\it Lipschitz continuous} in $C$ with constant $L>0$ if $\|\nabla f(x)-\nabla f(y)\|\leq L \|x-y\|$, for all~$x, y\in C$. Combining this definition with the fundamental theorem of calculus, we obtain the following result whose proof can be found in \cite[Proposition A.24]{Bertsekas1999}.

\begin{lemma} \label{Le:derivlipsch}
	Let $f:\mathbb{R}^n \to \mathbb{R}$ be a differentiable function and $C \subseteq \mathbb{R}^n$. Assume that $\nabla f$  is Lipschitz continuous in C with constant $L>0$. Then, $f(y) - f(x) - \langle \nabla f(x), y-x \rangle \leq (L/2)\|x-y\|^2$,  for all~ $x, y\in C$.
\end{lemma}
Assume that $C$ is a convex set. The function $f$ is said to be convex on $C$, if $f(y) \geq  f(x) + \langle \nabla f(x), y-x \rangle$, for all $x, y\in C$.
We recall  that a point ${\bar{x}} \in C$ is a {\it stationary point} for problem \eqref{eq:OptP} if
\begin{equation} \label{eq:StatPoint}
	\langle \nabla f({\bar{x}}),  x-{\bar{x}}\rangle \geq 0, \qquad \forall ~ x\in  C.
\end{equation}
Consequently, if $f$ is a convex function  on $C$, then  \eqref{eq:StatPoint} implies that  $\bar{x} \in \Omega^*$.  We end this section with some useful concepts  for the analysis of the sequence generated by the scaled  gradient method, for more details, see \cite{CombettesVu2013}.    For that,  let  $D$ be a $n\times n$ positive definite matrix and $\| \cdot \|_{D} : \mathbb{R}^{n}\rightarrow \mathbb{R}$ be  the norm  defined by
\begin{equation} \label{def:normaD}
	\|d\|_{D}:=\sqrt{\left\langle D d,d\right\rangle},\quad \forall d\in \mathbb{R}^{n}.
\end{equation}
For a fixed  constant $\mu \geq 1$,  {\it denote by  ${\cal D}_{\mu}$  the set of symmetric positive definite matrices $n\times n$ with all eigenvalues contained in the interval $[\frac{1}{\mu}, \mu]$}.  The set ${\cal D}_{\mu}$   is compact. Moreover,  for each $D\in {\cal D}_{\mu}$, it follows that $D^{-1}$ also belongs to $ {\cal D}_{\mu}$. Furthermore,  due to $D\in {\cal D}_{\mu}$,  by \eqref{def:normaD}, we obtain
\begin{equation} \label{eq:pnv}
	\frac{1}{\mu}\|d\|^2\leq \|d\|^2_{D}\leq \mu \|d\|^2, \qquad \forall d\in \mathbb{R}^n.
\end{equation}

\begin{definition} \label{def:QuasiFejer}
	Let $(y^k)_{k\in\mathbb{N}}$ be a sequence in $\mathbb{R}^n$ and   $(D_k)_{k\in\mathbb{N}}$ be  a sequence in ${\cal D}_{\mu}$.  The sequence $(y^k)_{k\in\mathbb{N}}$ is said to be quasi-Fej\'er convergent to a set $W\subset \mathbb{R}^n$ with respect to  $(D_k)_{k\in\mathbb{N}}$ if, for  all $w\in W$, there exists a sequence $(\epsilon_k)_{k\in\mathbb{N}}\subset\mathbb{R}$ such that $\epsilon_k\geq 0$, $\sum_{k\in \mathbb{N}}\epsilon_k<\infty$, and $\|y_{k+1}-w\|_{D_{k+1}}^2\leq \|y^k-w\|_{D_k}^2+\epsilon_k$, for    all $k\in \mathbb{N}$.
\end{definition}

The main property of a quasi-Fej\'er convergent sequence is stated in the next result. Its proof can be found in \cite{CombettesVu2013} but, for sake of completeness, we include it here.

\begin{theorem}\label{teo.qf}
	Let $(y^k)_{k\in\mathbb{N}}$ be a sequence in $\mathbb{R}^n$ and   $(D_k)_{k\in\mathbb{N}}$ be  a sequence in ${\cal D}_{\mu}$.   If $(y^k)_{k\in\mathbb{N}}$ is quasi-Fej\'er convergent to a nomempty set $W\subset  \mathbb{R}^n$ with respect to $(D_k)_{k\in\mathbb{N}}$, then $(y^k)_{k\in\mathbb{N}}$ is bounded. Furthermore, if a cluster point ${\bar y}$ of $(y^k)_{k\in\mathbb{N}}$ belongs to $W$, then $\lim_{k\rightarrow\infty}y^k={\bar y}$.
\end{theorem}
\begin{proof}
	Take $w\in W$. Definition~\ref{def:QuasiFejer} implies that $\|y^{k}-w\|_{D_{k}}^2\leq \|y^0-w\|_{D_0}^2+\sum_{k\in \mathbb{N}}\epsilon_k<+\infty$, for all $k\in  \mathbb{N}$.  Thus, using the  first  inequality in \eqref{eq:pnv}, we conclude that  $\|y^{k}-w\|\leq \sqrt{\mu} \|y^{k}-w\|_{D_{k}}$, for all $k\in  \mathbb{N}$.  Therefore, combining the two previous inequalities, we conclude that $(y^k)_{k\in\mathbb{N}}$ is bounded.  Let ${\bar y}\in W$ be a cluster point  of $(y^k)_{k\in\mathbb{N}}$ and  $(y^{k_j})_{j\in\mathbb{N}}$ be a subsequence of $(y^k)_{k\in\mathbb{N}}$ such that $\lim_{j\to +\infty} y^{k_j} = {\bar y}$. Take $\delta>0$. Since $\lim_{j\to +\infty} y^{k_j} = {\bar y}$ and $\sum_{j\in \mathbb{N}}\epsilon_k<\infty$,  there exists $j_0$ such that $\sum_{j\geq {j_0}}\epsilon_j<\delta /(2\mu)$ and $j_1>j_0$ such that $\|y^{k_j} - {\bar y}\|\leq \sqrt{\delta/2\mu^2} $, for all $j\geq j_1$. Hence, using  the first  inequality in \eqref{eq:pnv} and taking into account  that $\|y^{k+1}-{\bar y}\|_{D_{k+1}}^2\leq \|y^k-{\bar y}\|_{D_k}^2+\epsilon_k$,  for all $k\in \mathbb{N}$, we   have
	$
		\|y^{k} - {\bar y}\|^2\leq \mu \|y^{k} - {\bar y}\|_{D_k}^2\leq \mu(\|y^{k_j}-{\bar y}\|_{D_{k_j}}^2+\sum_{\ell=k_j}^{k-1}\epsilon_\ell),
	$
	for all $k\geq j_1$. Hence,  using  the second  inequality in \eqref{eq:pnv},  we conclude that $\|y^{k} - {\bar y}\|^2\leq \mu \|y^{k} - {\bar y}\|_{D_k}^2\leq \mu (\mu \|y^{k_j}-{\bar y}\|^2+\sum_{\ell=k_j}^{k-1}\epsilon_\ell)<\mu ( \frac{ \delta}{2\mu} + \frac{  \delta }{2\mu} )=\delta, $ for all  $k\geq j_1$.  Therefore, $\lim_{k\rightarrow\infty}y^k={\bar y}$.
\end{proof}

\subsection{Relative feasible inexact projections} \label{Sec:SubInexProj}

In this section, we recall two concepts  of relative feasible inexact projections onto a closed and convex set, and  also  present  some  new properties of them which will be used throughout the paper. These  concepts  of inexact projections were    introduced seeking to make the subproblem of computing the projections on the feasible  set more efficient;  see for example \cite{BirginMartinezRaydan2003,SalzoVilla2012,VillaSalzo2013}. Before presenting the  inexact projection concept that we will use, let us first recall the concept of exact projection with respect to a given  norm.  For that, {\it throughout this section  $D\in {\cal D}_{\mu}$}. The {\it exact  projection of the point $v\in \mathbb{R}^{n}$ onto $C$ with respect to the norm $\| \cdot \| _{D}$}, denoted by  ${\cal P}_{C}^{D}(v)$, is  defined~by
\begin{equation}\label{eq:exactM}
	{\cal P}_{C}^{D}(v):=\arg \min _{z\in C}\|z-v\|^2_{D}.
\end{equation}
The next result  characterizes  the exact projection, its  proof can be found in  \cite[Theorem 3.14]{BauschkeLivro2014}.

\begin{lemma} \label{pr:cham}
	Let $v, w \in {\mathbb R}^n$.  Then,  $w={\cal P}_{C}^{D}(v)$ if and only if  $w\in C$ and  $\left\langle D(v-w), y-w\right\rangle \leq  0$,   for all $y \in C.$
\end{lemma}

\begin{remark} \label{re:cproj}
	It follows from Lemma~\ref{pr:cham}  that $\|{\cal P}_{C}^{D}(v)-{\cal P}_{C}^{D}(u)\|_{D}\leq \|v-u\|_{D}$.  Moreover, since   $D\in {\cal D}_{\mu}$, by \eqref{eq:pnv}, we conclude that   ${\cal P}_{C}^{D}(\cdot)$ is Lipschitz continuous with constant $L=\mu$.   Furthermore,  if  $(D_k)_{k\in\mathbb{N}}\subset {\cal D}_{\mu}$,    $\lim_{k\to +\infty} z^{k} = \bar{z}$, and   $\lim_{k \to +\infty} D_{k} = \bar{D}$, then $\lim_{k \to +\infty}{\cal P}_{C}^{D_k}(z^{k})= {\cal P}_{C}^{\bar D}(\bar{z})$, see   \cite[Proposition~4.2]{CombettesVu2013}.
\end{remark}

In the following, we recall  the  concept of a  feasible inexact projection with respect to $\| \cdot \| _{D}$ relative to a fixed point. 

\begin{definition} \label{def:InexactM}
	The {\it feasible inexact projection mapping, with respect to the norm $\| \cdot \|_{D}$,   onto $C$}  relative to a point  $u \in C$ and forcing parameter $\zeta\in (0, 1]$, denoted by ${\cal P}_{C,\zeta}^{D}(u,  \cdot): {\mathbb R}^n \rightrightarrows C$,  is the set-valued mapping defined as follows
	\begin{equation} \label{eq:Projwm}
		{\cal P}_{C,\zeta}^{D}(u, v) := \left\{w\in C:~ \|w-v\|_{D}^2\leq \zeta \| {\cal P}_{C}^{D}(v)-v\|_{D}^2+(1-\zeta)\|u-v\|_{D}^2 \right\}.
	\end{equation}
	Each point $w\in {\cal P}_{C,\zeta}^{D}(u, v) $ is called a  feasible inexact projection,  with respect to the norm $\| \cdot \|_{D}$,  of $v$ onto $C$ relative to $u$ and forcing parameter $\zeta\in (0, 1]$.
\end{definition}

In the following, we show that the definition given above is nothing more than a reformulation of the concept of  relative feasible inexact projection with respect to $\| \cdot \|_{D}$  introduced in  \cite{BirginMartinezRaydan2003}.
\begin{remark}
	Let $u\in C$, $v\in \mathbb{R}^n$ and  $D$ be   an $n\times n$ positive definite matrix. Consider the quadratic  function $Q: \mathbb{R}^n \to \mathbb{R}$ defined by $Q(z):=(1/2) \left\langle {D}(z-u),z-u\right\rangle +  \left \langle D(u-v), z-u \right\rangle$.
	Thus,  letting  $\| \cdot \|_{D}$  be  the norm  defined by \eqref{def:normaD},  some algebraic manipulations   shows that
	\begin{equation} \label{eq:qppq}
		\|z-v\|^2_{D}= 2Q(z) +\|u-v\|^2_{D}.
	\end{equation}
	Hence,  \eqref{eq:qppq}  and \eqref{eq:exactM}  implies  that    ${\cal P}_{C}^{D}(v)=\arg \min _{z\in C}Q(z).$
	Let $\zeta\in (0, 1]$. Thus, by using \eqref{eq:qppq},  after some calculations,  we can see that  the following inexactness condition  introduced in \cite{BirginMartinezRaydan2003},
	$$
		w\in C, \qquad Q(w)\leq \zeta Q ( {\cal P}_{C}^{D}(v)) ,
	$$
	is  equivalent to find  $w\in C$ such that $\|w-v\|_{D}^2\leq \zeta \| {\cal P}_{C}^{D}(v)-v\|_{D}^2+(1-\zeta)\|u-v\|_{D}^2$,  which corresponds to condition \eqref{eq:Projwm} in Definition~\ref{def:InexactM}.
\end{remark}
The  concept of  feasible inexact projection  in Definition~\ref{def:InexactM}  provides  more latitude to   the usual  concept  of exact projection \eqref{eq:exactM}.  The next   remark makes  this more precise.
\begin{remark}\label{rem: welldef}
	Let $\zeta$ be positive forcing parameter, $C\subset {\mathbb R}^n$ and $u\in C$ be as in Definition~\ref{def:InexactM}.  First of all note that  ${\cal P}_{C}^{D}( v) \in {\cal P}_{C,\zeta}^{D}(u, v)$. Therefore,  ${\cal P}_{C,\zeta}^{D}(u, v)\neq \varnothing$, for all $u\in C$ and $v\in {\mathbb R}^n$. Consequently, the set-valued mapping ${\cal P}_{C,\zeta}^{D}(u,  \cdot)$ as stated in \eqref{eq:Projwm} is well-defined.   Moreover,  for $\zeta=1$, we have ${\cal P}_{C,1}^{D}(u, v)=\{{\cal P}_{C}^{D}(v)\}$.
	In addition, if $\underline{\zeta}$ and $\bar{\zeta}$ are forcing parameters such that $0<\underline{\zeta}\leq \bar{\zeta}\leq 1$, then ${\cal P}_{C,\bar{\zeta}}^{D}(u, v) \subset {\cal P}_{C,\underline{\zeta}}^{D}(u, v)$.
\end{remark}
\begin{lemma} \label{pr:condm}
	Let $v \in {\mathbb R}^n$, $u \in C$ and $w\in {\cal P}_{C,\zeta}^{D}(u, v)$. Then, there hold
	$$
		\left\langle D(v-w), y-w\right\rangle \leq  \frac{1}{2} \|w-y\|_{D}^2 +   \frac{1}{2} \left[\zeta \| {\cal P}_{C}^{D}(v)-v\|_{D}^2+(1-\zeta)\|u-v\|_{D}^2 - \|y-v\|_{D}^2\right] ,   \qquad y \in C.
	$$
\end{lemma}
\begin{proof}
	Let  $y \in C$. Since   $  2\langle D(v-w), y-w\rangle = \|w-y\|_{D}^2 + \|w-v\|_{D}^2-\|v-y\|_{D}^2$,  using \eqref{eq:Projwm}  we have
	$2\langle D(v-w), y-w\rangle = \|w-y\|_{D}^2 + \zeta \| {\cal P}_{C}^{D}(v)-v\|_{D}^2+(1-\zeta)\|u-v\|_{D}^2-\|v-y\|_{D}^2$, which is equivalent to  the desired inequality.
\end{proof}
Next, we recall a second  concept of relative  feasible inexact projection onto a closed convex set, see  \cite{Ademir_Orizon_Leandro2020, OrizonFabianaGilson2018}.  The definition  is as follows.
\begin{definition} \label{def:InexactProjC}
	The {\it feasible inexact projection mapping, with respect to the norm $\| \cdot \|_{D}$,  onto $C$} relative to $u \in C$ and forcing parameter $\gamma\geq 0$, denoted by ${\cal R}_{C,\gamma}^{D}(u, \cdot): {\mathbb R}^n \rightrightarrows C$,  is the set-valued mapping defined as follows
	\begin{equation} \label{eq:Projw}
		{\cal R}_{C,\gamma}^{D}(u, v):= \left\{w\in C:~\left\langle D(v-w), y-w \right\rangle \leq \gamma \|w-u\|_{D}^2, \quad \forall~ y \in C \right\}.
	\end{equation}
	Each point $w\in {\cal R}_{C,\gamma}^{D}(u, v)$ is called a feasible inexact projection,  with respect to the norm $\| \cdot \|_{D}$,  of $v$ onto $C$ relative to $u$ and forcing parameter $\gamma\geq 0$.
\end{definition}
The concept of  feasible inexact projection  in Definition~\ref{def:InexactProjC} also  provides  more latitude to  the usual concept  of exact projection. Next,  we present some remarks about this concept.
\begin{remark}\label{rem: welldef}
	Let $\gamma\geq 0$ be a forcing parameter, $C\subset {\mathbb R}^n$ and $u\in C$ be as in Definition~\ref{def:InexactProjC}.
	For all $v\in {\mathbb R}^n$, it follows from \eqref{eq:Projw} and Lemma~\ref{pr:cham} that ${\cal R}_{C,0}^{D}(u, v)=\{{\cal P}_{C}^{D}(v)\}$ is the exact projection of $v$ onto $C$. Moreover, ${\cal P}_{C}^{D}(v)\in {\cal R}_{C,\gamma}^{D}(u, v)$ concluding  that ${\cal R}_{C, \gamma}(u, v)\neq \varnothing$, for all $u\in C$ and $v\in {\mathbb R}^n$. Consequently, the set-valued mapping ${\cal R}_{C,\gamma}^{D}(u, \cdot)$ as stated in \eqref{eq:Projw} is well-defined.
\end{remark}
The  next lemma is a variation of \cite[Lemma 6]{Reiner_Orizon_Leandro2019}.  It will allow to relate Definitions~\ref{def:InexactM} and \ref{def:InexactProjC}.
\begin{lemma} \label{pr:cond}
	Let $v \in {\mathbb R}^n$, $u \in C$, $\gamma \geq 0$ and $w\in {\cal R}_{C,\gamma}^{D}(u, v)$. Then, there hold
	$$
		\displaystyle \|w-x\|_{D}^2 \leq \|x-v\|_{D}^2 + \frac{2\gamma}{1-2\gamma}\|u-v\|_{D}^2- \frac{1}{1-2\gamma}\|w-v\|_{D}^2,
	$$
	for all $x \in C$ and $0 \leq \gamma <1/2$.
\end{lemma}
\begin{proof}
	First note that $\|w-x\|_{D}^2 = \|x-v\|_{D}^2 - \|w-v\|_{D}^2 + 2 \langle D(v-w), x-w \rangle$. Since $w \in {\cal R}_{C,\gamma}^{D}(u, v)$ and $x \in C$, combining the last equality with \eqref{eq:Projw}, we obtain
	\begin{equation} \label{eq:fg}
		\|w-x\|_{D}^2 \leq \|x-v\|_{D}^2 - \|w-v\|_{D}^2  + 2\gamma \|w-u\|_{D}^2.
	\end{equation}
	On the other hand, we also have $\|w-u\|_{D}^2=\|u-v\|_{D}^2 - \|w-v\|_{D}^2 +  2 \langle D(v-w), u-w \rangle$. Due to $w\in {\cal R}_{C,\gamma}^{D}(u, v)$ and $u \in C$, using \eqref{eq:Projw} and considering that  $0 \leq \gamma < 1/2$, we have
	$$
		\|w-u\|_{D}^2 \leq \frac{1}{1-2\gamma}\|u-v\|_{D}^2 - \frac{1}{1-2\gamma} \|w-v\|_{D}^2.
	$$
	Therefore, substituting the last inequality into   \eqref{eq:fg}, we obtain the  desired inequality.
\end{proof}
In the following  lemma, we present a  relationship between   Definitions~\ref{def:InexactM} and \ref{def:InexactProjC}.
\begin{lemma} \label{pr:condrip}
	Let $v \in {\mathbb R}^n$, $u \in C$, $\gamma \geq 0$  and $\zeta\in (0, 1]$.  If  $0 \leq \gamma <1/2$ and $\zeta=1-2\gamma$, then
	$$
		{\cal R}_{C,\gamma}^{D}(u, v) \subset {\cal P}_{C,\zeta}^{D}(u, v).
	$$
\end{lemma}
\begin{proof}
	Let $w\in {\cal R}_{C,\gamma}^{D}(u, v)$. Applying Lemma~\ref{pr:cond} with  $x={\cal P}_{C}^{D}(v)$ we have
	$$
		\|w-{\cal P}_{C}^{D}(v)\|_{D}^2 \leq \|v-{\cal P}_{C}^{D}(v)\|_{D}^2 + \frac{2\gamma}{1-2\gamma}\|u-v\|_{D}^2- \frac{1}{1-2\gamma}\|w-v\|_{D}^2,
	$$
	After some algebraic manipulations in the last inequality we obtain that
	$$
		\|w-v\|_{D}^2 \leq (1-2\gamma)\|v-{\cal P}_{C}^{D}(v)\|_{D}^2 + 2\gamma\|u-v\|_{D}^2- (1-2\gamma)\|w-{\cal P}_{C}^{D}(v)\|_{D}^2.
	$$
	Therefore, considering that  $0 \leq \gamma <1/2$ and $\zeta=1-2\gamma$, the result follows from Definition~\ref{def:InexactM}.
\end{proof}

\begin{remark} \label{re:ni}
	Under the conditions of  Lemma \ref{pr:condrip}, there exists  $0 \leq \gamma <1/2$ and $\zeta=1-2\gamma$ such that ${\cal P}_{C,\zeta}^{D}(u, v)  \nsubseteq    {\cal R}_{C,\gamma}^{D}(u, v)$. Indeed,  let $\gamma=3/8$, $\zeta=1/4$,  and ${\bar w}=\frac{1}{2}({\cal P}_{C}^{D}(v)+u)$, then
	$$
		\displaystyle\|{\bar w}-v\|_D^2=\frac{1}{4}\| {\cal P}_{C}^{D}(v)-v\|_D^2 + \frac{1}{4}\|u-v\|_D^2 +\frac{1}{2} \langle D( {\cal P}_{C}^{D}(v) -v), u-v \rangle.
	$$
	Since  $ {\cal P}_{C}^{D}(v)$ is the exact projection of $v$,  we have   $\displaystyle\langle D( {\cal P}_{C}^{D}(v) -v), u-v \rangle \leq \|u-v\|_D^2$. Combining this inequality with  the last equality and Definition~\ref{def:InexactM}, we conclude that ${\bar w}\in {\cal P}_{C,\zeta}^{D}(u, v)$. Now,  letting $w_t=t{\cal P}_{C}^{D}(v)+ (1-t){\bar w}$  with $0<t<1$, after some algebraic  manipulations  we have
	$$
		\langle D(v-{\bar w}), w_t-{\bar w} \rangle=t\| {\bar w}-u\|_D^2 - \frac{t}{2} \langle D( v- {\cal P}_{C}^{D}(v)) , u-{\cal P}_{C}^{D}(v)  \rangle.
	$$
	Thus, it follows from Lemma~\ref{pr:cham} that $\displaystyle \langle D(v-{\bar w}), w_t-{\bar w} \rangle\geq t\| {\bar w}-u\|_D^2 $.  Hence,  taking  $t>3/8$ we conclude that ${\bar w}\not\in{\cal R}_{C,\gamma}^{D}(u, v)$.  Therefore, considering that ${\bar w}\in {\cal P}_{C,\zeta}^{D}(u, v)$, the statement follows.
\end{remark}
It follows from Remark~\ref{re:ni} that, in general,  ${\cal P}_{C,\zeta}^{D}(u, v)  \nsubseteq    {\cal R}_{C,\gamma}^{D}(u, v)$. However, whenever $C$ is a bounded set,  we will show  that   for each  fixed  $0 \leq \gamma <1/2$  there exist $0 < \zeta  <1$ such that    ${\cal P}_{C,\zeta}^{D}(u, v)  \subseteq    {\cal R}_{C,\gamma}^{D}(u, v)$. For that, we first need the next lemma.
\begin{lemma} \label{le:epsi}
	Let $v \in {\mathbb R}^n$, $u \in C$ and $0<\gamma < 1/2$. Assume that $C$ is a bounded set and take
	\begin{equation} \label{eq:epsi}
		0<\varepsilon < \frac{\gamma \|u-\mathcal{P}_C^D(v)\|_D^2}{1-\gamma+\|v-\mathcal{P}_C^D(v)\|_D+ 2	\gamma\|u-\mathcal{P}_C^D(v)\|_D+\mbox{\normalfont diam}C },
	\end{equation}
	where $\mbox{\normalfont diam}C$ denotes the diameter of $C$. Then, $\{w\in C: ~ \|w-\mathcal{P}_C^D(v)\|_D\leq \varepsilon\}\subset\mathcal{R}_{C,\gamma}^D(u, v)\}$.
\end{lemma}
\begin{proof}
	Take $\varepsilon$ satisfying \eqref{eq:epsi} and $w\in C$ such that  $\|w-\mathcal{P}_C^D(v)\|_D\leq \varepsilon$ . For all $z\in C$, we have
	\begin{multline*}
		\langle D(v-w), z-w \rangle = \langle D(v-\mathcal{P}_C^D(v)),z-\mathcal{P}_C^D(v) \rangle + \langle D(v-\mathcal{P}_C^D(v)), \mathcal{P}_C^D(v) -w \rangle\\
		+ \langle D(\mathcal{P}_C^D(v) - w), z- \mathcal{P}_C^D(v) \rangle +\| \mathcal{P}_C^D(v)-w\|_D^2.
	\end{multline*}
	Using Lemma~\ref{pr:cham},  we have $\langle D(v-\mathcal{P}_C^D(v)),z-\mathcal{P}_C^D(v) \rangle\leq 0$. Thus, the last equality becomes
	$$
		\langle D(v-w), z-w \rangle \leq \langle D(v-\mathcal{P}_C^D(v)), \mathcal{P}_C^D(v) -w \rangle + \langle D(\mathcal{P}_C^D(v) - w), z- \mathcal{P}_C^D(v) \rangle +\| \mathcal{P}_C^D(v)-w\|_D^2.
	$$
	By using Cauchy-Schwarz inequality, we conclude from the last inequality that
	$$
		\langle D(v-w), z-w \rangle \leq  \|w-\mathcal{P}_C^D(v) \|_D\left(\|v-\mathcal{P}_C^D(v)\|_D+ \|z- \mathcal{P}_C^D(v)\|_D\right) +\|w-\mathcal{P}_C^D(v)\|_D^2.
	$$
	Since $\|w-\mathcal{P}_C^D(v)|_D\leq \varepsilon$ and $ \|z- \mathcal{P}_C^D(v)\|_D\leq \mbox{diam} C$, the last inequality implies that
	\begin{equation}\label{eq:diam1}
		\langle D(v-w), z-w \rangle \leq \varepsilon \left(\|v-\mathcal{P}_C^D(v)\|_D+\mbox{diam}C\right)+\varepsilon^2,
	\end{equation}
	On the other hand, if $\varepsilon$ satisfies \eqref{eq:epsi} then
	$$
		\varepsilon \left( 1-\gamma+\|v-\mathcal{P}_C^D(v)\|_D+ \mbox{diam} C\right) + \gamma \varepsilon^2 <
		\gamma\|u-\mathcal{P}_C^D(v)\|_D^2-2	\gamma\varepsilon\|u-\mathcal{P}_C^D(v)\|_D +\gamma\varepsilon^2,
	$$
	hence
	$$
		\varepsilon \left( 1-\gamma+\|v-\mathcal{P}_C^D(v)\|_D+ \mbox{diam} C\right) + \gamma \varepsilon^2 < \gamma\left(\|u-\mathcal{P}_C^D(v)\|_D-\varepsilon\right)^2.
	$$
	Since $\gamma, \varepsilon<1$, we have  $\varepsilon^2 < \varepsilon(1-\gamma) + \gamma \varepsilon^2$ and we can conclude that
	$$
		\varepsilon \left( \|v-\mathcal{P}_C^D(v)\|_D+ \mbox{diam} C\right) + \varepsilon^2 <
		\gamma\left(\|u-\mathcal{P}_C^D(v)\|_D-\varepsilon\right)^2.
	$$
	It follows from \eqref{eq:diam1} that
	\begin{equation}\label{eq:diam2}
		\langle D(v-w), z-w \rangle \leq \gamma\left(\|u-\mathcal{P}_C^D(v)\|_D-\varepsilon\right)^2.
	\end{equation}
	Using again that $\|w-\mathcal{P}_C^D(v)|_D\leq \varepsilon$ and the triangular inequality, we have
	$$
		0<\|u-\mathcal{P}_C^D(v)\|_D -\varepsilon \leq \|u-\mathcal{P}_C^D(v)\|_D -\|w-\mathcal{P}_C^D(v)\|_D \leq \|u-w\|_D.
	$$
	Hence,   taking into account  \eqref{eq:diam2}, we conclude that $	\langle D(v-w), z-w \rangle \leq \gamma\|u-w\|_D^2$. Therefore, it follows from Definition~\ref{def:InexactProjC} that    $w\in\mathcal{R}_{C,\gamma}^D(u, v)$.
\end{proof}
\begin{proposition} \label{le:pcr}
	Let $v \in {\mathbb R}^n$, $u \in C$ and assume that $C$ is a bounded set. Then, for each $0<\gamma < 1/2$,     there exist $0 < \zeta  <1$ such that    ${\cal P}_{C,\zeta}^{D}(u, v)  \subseteq    {\cal R}_{C,\gamma}^{D}(u, v)$.
\end{proposition}
\begin{proof}
	It follows from Lemma~ \ref{le:epsi}  that given $0<\gamma<1/2$ there exists $\varepsilon>0$ such that,  for all $w\in C$  with $\|w-\mathcal{P}_C^D(v)\|\leq \varepsilon$, we have  $w\in\mathcal{R}_\gamma^D(v)$. Otherwise,  we can see in (\ref{eq:Projwm}), when $\zeta\to 1$, the diameter of $C\cap{\cal P}_{C,\zeta}^{D}(u, v)$ tends to zero, then there exists  $\zeta$ close to 1 such that $\mbox{diam}(C\cap{\cal P}_{C,\zeta}^{D}(u, v))<\varepsilon/2$, and ${\cal P}_{C,\zeta}^{D}(u, v) \subset {\cal R}_{C,\gamma}^{D}(u, v)$.
\end{proof}
Next, we present some important properties of  inexact projections, which will be useful in  the sequel.
\begin{lemma} \label{Le:ProjProperty}
	Let $x \in C$, $\alpha > 0$ and  $z(\alpha) = x-\alpha D^{-1} \nabla f(x)$. Take $w(\alpha) \in  {\cal P}_{C,\zeta}^{D}(x, z(\alpha))$ with $\zeta\in (0, 1]$. Then, there hold
	\begin{itemize}
		\item[(i)] $\displaystyle \langle \nabla f(x), w(\alpha) - x\rangle \leq -\frac{1}{2\alpha} \|w(\alpha) -x\|_{D}^2 +   \frac{\zeta}{2\alpha} \left[\| {\cal P}_{C}^{D}(z(\alpha))-z(\alpha)\|_{D}^2 - \|x-z(\alpha)\|_{D}^2\right]$;
		\item[(ii)] the point $x$ is stationary for problem \eqref{eq:OptP} if and only if $x \in {\cal P}_{C,\zeta}^{D}(x, z(\alpha))$;
		\item[(iii)] if  $x \in C$ is a nonstationary point for problem \eqref{eq:OptP}, then $\Big\langle \nabla f(x), w(\alpha) - x \Big\rangle < 0$. Equivalently, if there exists ${\bar \alpha}>0$ such that $\Big\langle \nabla f(x), w({\bar \alpha}) - x \Big\rangle \geq 0$, then $x$ is stationary for problem \eqref{eq:OptP}.
	\end{itemize}
\end{lemma}
\begin{proof}
	Since $w(\alpha) \in  {\cal P}_{C,\zeta}^{D}(x, z(\alpha))$, applying Lemma~\ref{pr:condm} with   $w=w(\alpha)$, $v=z(\alpha)$, $y=x$,  and $u=x$, we conclude,  after some algebraic manipulations,  that
	$$
		\left\langle D(z(\alpha)-w(\alpha)), x-w(\alpha) \right\rangle \leq 	\frac{1}{2} \|w(\alpha) -x\|_{D}^2 +   \frac{\zeta}{2} \left[\| {\cal P}_{C}^{D}(z(\alpha))-z(\alpha)\|_{D}^2 - \|x-z(\alpha)\|_{D}^2\right].
	$$
	Substituting  $z(\alpha) = x-\alpha \nabla f(x)$ in the left hand side of the last inequality,   some manipulations yield the inequality of item $(i)$.  For proving item $(ii)$, we first assume that $x$ is stationary for problem \eqref{eq:OptP}. In this case, \eqref{eq:StatPoint} implies that $\langle \nabla f(x), w(\alpha)-x \rangle \geq 0$. Hence,  due to $\|{\cal P}_{C}^{D}(z(\alpha))-z(\alpha)\|_{D}\leq  \|x-z(\alpha)\|_{D}$,   item $(i)$ implies
	$$
		\frac{1}{2\alpha} \|w(\alpha) -x\|_{D}^2 \leq    \frac{\zeta}{2\alpha} \left[\| {\cal P}_{C}^{D}(z(\alpha))-z(\alpha)\|_{D}^2 - \|x-z(\alpha)\|_{D}^2\right]\leq 0.
	$$
	Since $\alpha > 0$ and  $\zeta\in (0, 1]$, the last inequality  yields  $w(\alpha) = x$.  Therefore, $x \in {\cal P}_{C,\zeta}^{D}(x, z(\alpha))$. Reciprocally, if  $x \in {\cal P}_{C,\zeta}^{D}(x, z(\alpha))$, then Definition~\ref{def:InexactM} implies that
	$$
		\|x-z(\alpha)\|_{D}^2\leq \zeta \| {\cal P}_{C}^{D}(z(\alpha))-z(\alpha)\|_{D}^2+(1-\zeta)\|x-z(\alpha)\|_{D}^2.
	$$
	Hence, $0\leq \zeta\left( \| {\cal P}_{C}^{D}(z(\alpha))-z(\alpha)\|_{D}^2-(\|x-z(\alpha)\|_{D}^2\right)$. Considering that  $\zeta\in (0, 1]$ we have
	$$
		\|x-z(\alpha)\|_{D}\leq\| {\cal P}_{C}^{D}(z(\alpha))-z(\alpha)\|_{D}.
	$$
	Thus, due to  exact projection  with respect to the norm $\| \cdot \| _{D}$   be unique and $z(\alpha) = x-D^{-1} \alpha \nabla f(x)$,   we have    ${\cal P}_{C}^{D}( x-\alpha D^{-1} \nabla f(x))=x$. Hence, $x$ is the  solution of the constrained optimization problem $  \min _{y\in C}  \|y-z(\alpha)\| ^{2}_{D}$,
	which taking into account that  $\alpha > 0$ implies \eqref{eq:StatPoint}. Therefore, $x$ is stationary point for problem \eqref{eq:OptP}. Finally, to prove item $(iii)$, take $x$ a nonstationary point for problem \eqref{eq:OptP}. Thus, by item $(ii)$, $x \notin  {\cal P}_{C,\zeta}^{D}(x, z(\alpha))$ and taking into account that $w(\alpha) \in  {\cal P}_{C,\zeta}^{D}(x, z(\alpha))$, we conclude that $x \neq w(\alpha)$. Since $\|{\cal P}_{C}^{D}(z(\alpha))-z(\alpha)\|_{D}\leq  \|x-z(\alpha)\|_{D}$, $\alpha > 0$ and $\zeta\in (0, 1]$, it follows from item $(i)$ that $\Big\langle \nabla f(x), w(\alpha) - x \Big\rangle < 0$ and the first sentence is proved. Finally, note that the second sentence is the contrapositive of the first sentence.
\end{proof}
Finally, it is worth mentioning that Definitions~\ref{def:InexactM} and \ref{def:InexactProjC}, introduced respectively in \cite{BirginMartinezRaydan2003} and \cite{OrizonFabianaGilson2018},  are relative inexact  concepts, while the concept introduced in \cite{SalzoVilla2012,VillaSalzo2013} is absolute.

\subsubsection{Practical computation of  inexact projections} \label{sec:pcip}

In this section,  for a given  $v\in {\mathbb R}^n$ and   $u \in C$, we discuss how to calculate a  point  $w \in C$ belonging to  ${\cal P}_{C, \zeta}^{D}(u,v)$ or ${\cal R}_{C, \gamma}^{D}(u,v)$.  We recall that  Lemma~\ref{pr:condrip} implies that ${\cal P}_{C, \zeta}^{D}(u,v)$ has  more latitude than ${\cal R}_{C, \gamma}^{D}(u,v)$, i.e., ${\cal R}_{C, \gamma}^{D}(u,v) \subset {\cal P}_{C, \zeta}^{D}(u,v)$.

We begin our discussion by showing how a point $w\in {\cal P}_{C, \zeta}^{D}(u,v)$ can be calculated  without knowing the point ${\cal P}_{C}^{D}(v)$.   Considering that this discussion has already been covered in \cite[Section~3,  Algorithm 3.1]{BirginMartinezRaydan2003},  we will limit ourselves to giving a general idea of how this task is carried out; see also \cite[Section~5.1]{Bonettini2016}.   The idea is to use an external procedure capable of computing two  sequences  $(c_\ell)_{\ell\in\mathbb{N}}\subset \mathbb{R}$  and   $(w^\ell)_{\ell\in\mathbb{N}}\subset C$ satisfying the following  conditions
\begin{equation}\label{def:cl}
	c_\ell \leq \|{\cal P}_{C}^{D}(v)-v\|_D^2,\quad  \forall\ell\in\mathbb{N}, \qquad \qquad \lim_{\ell \to +\infty} c_\ell = \|{\cal P}_{C}^{D}(v)-v\|_D^2, \qquad \quad   \lim_{\ell\to +\infty} w^\ell = {\cal P}_{C}^{D}(v).
\end{equation}
In this case,   if  $v\notin C$, then   we have $ \|{\cal P}_{C}^{D}(v)-v\|_D^2- \|u-v\|_D^2<0$. Hence,   given an arbitrary $ \zeta\in (0,1)$,  the second condition in \eqref{def:cl} implies  that   there exists $\hat{\ell}$ such that
$$
	\|{\cal P}_{C}^{D}(v)-v\|_D^2- \|u-v\|_D^2< \zeta (c_{\hat{\ell}}- \|u-v\|_D^2).
$$
Moreover, by using the last condition in \eqref{def:cl}, we conclude that   there exists $\bar{\ell}>\hat{\ell}$ such that
\begin{equation}\label{def:clsc}
	\|w_{\bar{\ell}}-v\|_D^2- \|u-v\|_D^2< \zeta (c_{\bar{\ell}}- \|u-v\|_D^2),
\end{equation}
which using   the inequality  in \eqref{def:cl} yields  $ \ \|w_{\bar{\ell}}-v\|_D^2<  \zeta \|{\cal P}_{C}^{D}(v)-v\|_D^2  +(1-\zeta) \|u-v\|_D^2$.
Hence, Definition~\ref{def:InexactM} implies that  $w_{\bar{\ell}}\in  {\cal P}_{C, \zeta}^{D}(u, v)$.   Therefore,   \eqref{def:clsc}  can be used as a stopping criterion to compute  a  feasible inexact projection,  with respect to the norm $\| \cdot \|_{D}$,  of $v$ onto $C$ relative to $u$ and forcing parameter $\zeta\in (0, 1]$. For instance,  it follows from   \cite[Theorem~3.2, Lemma~3.1]{BirginMartinezRaydan2003}  that  such sequences  $(c_\ell)_{\ell\in\mathbb{N}}\subset \mathbb{R}$  and   $(w^\ell)_{\ell\in\mathbb{N}}\subset C$  satisfying  \eqref{def:cl}  can be computed by using   {\it Dykstra's algorithm} \cite{Dykstra1986, Dykstra1983}, whenever  $D$ is the identity matrix  and the set $C=\cap_{i=1}^p C_i$, where $C_i$ are closed and convex sets and the exact projection  ${\cal P}_{C_i}^{D}(v)$ is easy to obtain, for all $i=1,\dots,p$.

We end this section by discussing how to compute a point $w\in{\cal R}_{C, \gamma}^{D}(u,v)$. For that, we apply the classical {\it Frank-Wolfe method}, also known as {\it conditional gradient method},   to minimize the function  $\psi(z) := \|z - v\|^2/2$ onto  the constraint set $C$ with  a suitable  stop criteria depending  of $u\in C$ and $\gamma \in (0, 1]$, see  \cite{BeckTeboulle2004, Jaggi2013}.  To state the method  we assume the existence of a linear optimization oracle (or simply LO oracle) capable of minimizing linear functions over the constraint set $C$,   which is assumed to be  compact. The   Frank-Wolfe method is  formally stated as follows.

\begin{algorithm}[H]
	\begin{description}
		\item[Input:]  $D\in {\cal D}_{\mu}$,  $\gamma \in (0, 1]$,  $v\in {\mathbb R}^n$ and   $u \in C$. 
		\item[ Step 0.] Let $w^0\in  C$ and  set $\ell \gets 0$.
		\item[ Step 1.] Use a LO oracle to compute an optimal solution $z^\ell$ and the optimal value $s_{\ell}^*$ as
			\begin{equation}\label{eq:CondG_{C}}
				z^\ell \in  \arg\min_{z \in  C} \,\langle w^\ell-v, ~z-w^\ell\rangle,  \qquad s_{\ell}^*:=\langle  w^\ell-v, ~z^\ell-w^\ell \rangle.
			\end{equation}
			If $-s^*_{\ell}\leq \gamma \|w^{\ell}-u\|_{D}^2$, then define $w:=w^\ell$ and {\bf stop}.
		\item[ Step 2.]  Compute $\alpha_\ell$ and $w_{\ell+1}$ as
			\begin{equation}\label{eq:step size}
				w_{\ell+1}:=w^\ell+ \alpha_\ell(z^\ell-w^\ell), \qquad {\alpha}_\ell: =\min\left\{1, -s^*_{\ell}/\|z^\ell-w^\ell\|^2 \right\}.
			\end{equation}
			Set $\ell\gets \ell+1$, and go to Step~1.
		\item[Output:]   $w:=w^\ell$.
	\end{description}
	\caption{{\bf: Frank-Wolfe  method to compute} $w\in {\cal R}_{C, \gamma}^{D}(u,v)$}
	\label{Alg:CondG}
\end{algorithm}

Let us  describe the main features of  Algorithm~\ref{Alg:CondG}, i.e.,  the  Frank-Wolfe method applied to the problem $\min_{z \in C}\psi(z)$.   In this case,  \eqref{eq:CondG_{C}} is equivalent to $s_{\ell}^*:=\min_{z \in C}\langle \psi'(w^\ell) ,~z-w^\ell\rangle$. Since $\psi$ is convex, we have $\psi(z)\geq \psi(w^\ell) + \langle \psi'(w^\ell) ,~z-w^\ell\rangle\geq    \psi(w^\ell)  +   s_{\ell}^*$, for all $z\in C$. Define  $ w_*:=\arg \min_{z \in C}\psi(z)$ and  $\psi^*:= \min_{z \in C}\psi(z)$. Letting $z= w_*$ in the last inequality,  we obtain $\psi(w^\ell)\geq \psi^* \geq \psi(w^\ell)  +   s_{\ell}^*$, which implies that $s_{\ell}^*< 0$ whenever $\psi(w^\ell)\neq \psi^*$. Thus, we conclude that  $-s_{\ell}^*=\langle  v-w^\ell, ~z^\ell-w^\ell \rangle>0\geq  \langle  v-w_*, ~z-w_* \rangle$,  for all~$z\in C$.   Therefore, if  Algorithm~\ref{Alg:CondG} computes  $w^\ell \in C$ satisfying $-s_{\ell}^*\leq  \gamma \|w^{\ell}-u\|_{D}^2$, then the method terminates. Otherwise, it computes the step size $\alpha_\ell = \arg\min_{\alpha \in [0,1]} \psi(w^\ell + \alpha(z^\ell - w^\ell))$  using exact minimization.  Since $z^\ell$, $w^\ell \in C$  and $C$ is convex, we conclude from  \eqref{eq:step size}  that $w_{\ell+1} \in C$, thus  Algorithm~\ref{Alg:CondG}  generates a sequence in $C$.  Finally,   \eqref{eq:CondG_{C}} implies that  $\langle  v-w^\ell, ~z-w^\ell\rangle\leq -s_{\ell}^*$, for all  $ z\in C$.  Considering that   \cite[Proposition A.2]{BeckTeboulle2004}  implies that $\lim_{\ell \to +\infty}s_{\ell}^*=0$ and taking into account   the  stopping criteria     $-s_{\ell}^*\leq \gamma \|w^{\ell}-u\|_{D}^2$, we conclude that the output of  Algorithm~\ref{Alg:CondG} is  a feasible inexact projection $w\in {\cal R}_{C, \gamma}^{D}(u,v)$ i.e.,   $\langle  v-w, ~z-w\rangle\leq   \gamma \|w^{\ell}-u\|_{D}^2$, for all $z\in C$.


\section{Inexact scaled gradient method} \label{Sec:SGM}

The aim of this section is to present an  inexact  version of the scaled gradient method (SGM), which   inexactness are in two distinct senses.  First,  we  use  a version of the inexactness scheme introduced in \cite{BirginMartinezRaydan2003},  and also a variation of the one appeared in \cite{VillaSalzo2013},  to compute an inexact projection  onto the feasible  set   allowing an appropriate  relative error tolerance. Second,  using the  inexactness  conceptual scheme for  non-monotones    line  search   introduced  in  \cite{GrapigliaSachs2017, SachsSachs2011}, a step size is computed  to define the next iterate.  The statement of the  conceptual algorithm is as follows.\\

\begin{algorithm}[H]
	\begin{description}
		 \item[Step 0.] Choose  $\sigma,{\zeta_{\min}}  \in (0, 1)$, $\delta_{\min}\in [0, 1)$, $0<\underline \omega<\bar \omega<1$, $0 < \alpha_{\min} \leq \alpha_{\max}$ and $\mu \geq1$. Let $x^0\in C$, $\nu_0\geq 0$ and set $k \gets 0$.
		 \item[Step 1.] Choose positive real numbers $\alpha_k$ and $\zeta_k$, and a positive definite matrix $D_k$ such that
		\begin{equation} \label{eq:TolArm}
			\alpha_{\min}\leq \alpha_k \leq \alpha_{\max}, \qquad \qquad 0 <{\zeta_{\min}}<\zeta_k \leq 1, \qquad \qquad D_k\in {\cal D}_{\mu}.
		\end{equation}
		Compute $w^{k}\in C$  as any feasible inexact projection  with respect to the norm $\| \cdot \| _{D_k}$ of $z^k := x^{k}-\alpha_k D_k^{-1}\nabla f(x^{k})$ onto $C$ relative to $x^{k}$  with forcing parameter $\zeta_k$, i.e.,
		\begin{equation} \label{eq:PInexArm}
			w^k \in   {\cal P}_{C, \zeta_k}^{D_k}(x^{k}, z^k).
		\end{equation}
		If $w^k= x^k$, then {\bf stop} declaring convergence.
		 \item[Step 2.]  Set $\tau_{\textrm{trial}} \gets 1$. If
		\begin{equation}\label{eq:TkArm}
			 f\big(x^{k}+ \tau_{\textrm{trial}}(w^k - x^{k})\big) \leq f(x^{k}) + \sigma \tau_{\textrm{trial}}\big\langle \nabla f(x^{k}), w^k - x^{k} \big\rangle + \nu_k,
		\end{equation}
		then  $\tau_k\gets \tau_{\textrm{trial}}$, define the next iterate $x^{k+1}$ as
		\begin{equation} \label{eq:IterArm}
			x^{k+1} = x^{k} + \tau_k (w^k - x^{k}),
		\end{equation}
		and go to {\bf Step 3}. Otherwise, choose $\tau_{\textrm{new}} \in [\underline\omega \tau_{\textrm{trial}}, \bar\omega \tau_{\textrm{trial}} ]$, set $\tau_{\textrm{trial}} \gets \tau_{\textrm{new}}$, and repeat test \eqref{eq:TkArm}.
		
		 \item[Step 3.]  Take  $\delta_{k+1}\in [\delta_{\min}, 1]$ and choose    $\nu_{k+1}\in {\mathbb R}$ satisfying
		\begin{equation} \label{eq:nuk}
			0\leq \nu_{k+1}\leq (1-\delta_{k+1})\big[f(x^{k})+\nu_{k}-f(x^{k+1})\big].
		\end{equation}
		Set $k\gets k+1$ and go to \textbf{Step~1}.
\end{description}
	
	\caption{ InexProj-SGM employing non-monotone line search}
	\label{Alg:GeneralSeach}
\end{algorithm}

Let us describe the main features of Algorithm~\ref{Alg:GeneralSeach}. In Step~1,  we first  choose   $\alpha_{\min}\leq \alpha_k \leq \alpha_{\max}$, $0 < \zeta_{\min} \leq \zeta_k  < 1$, and  $D_k\in  {\cal D}_{\mu}$. Then, by using some (inner) procedure, such as those specified in Section~\ref{Sec:SubInexProj}, we compute $w^k$ as any feasible inexact projection of $z^k = x_k - \alpha_kD_k^{-1}\nabla f(x_k)$ onto the feasible set $C$ relative to the previous iterate $x^k$ with forcing parameter $\zeta_k$. If $w^k= x^k$, then Lemma~\ref{Le:ProjProperty}{\it (ii)} implies that $x^{k}$ is a solution of  problem \eqref{eq:OptP}.  Otherwise,  $w^k\neq  x^k$ and Lemma~\ref{Le:ProjProperty}{\it (i)}  implies  that $ w^k- x^k$ is a descent direction of $f$ at $x^k$, i.e.,  $\langle \nabla f(x^k), w^k- x^k \rangle < 0$.    Hence, in Step~2, we employ a non-monotone line search  with tolerance parameter $\nu_k\geq 0$ to compute a step size  $\tau_k \in (0, 1]$,  and  the next iterate is computed as in \eqref{eq:IterArm}. Finally, due to  \eqref{eq:TkArm} and  $\delta_{k+1}\in [\delta_{\min}, 1]$, we have $0\leq (1-\delta_{k+1})\big[f(x^{k})+\nu_{k}-  f(x^{k+1})\big]$.  Therefore, the next   tolerance parameter $\nu_{k+1}\in {\mathbb R}$ can be chosen satisfying \eqref{eq:nuk}  in Step~3, completing the iteration.

 It is worth mentioning that the conditions in \eqref{eq:TolArm}  allow combining several strategies for choosing the step sizes $\alpha_k$  and the matrices $D_k$  to accelerate the performance of the classical gradient method.   Strategies  of choosing the step sizes $\alpha_k$  and the matrices $D_k$ have their origin in the study of the gradient  method  for unconstrained  optimization,  papers dealing with this issue include  but are not limited to \cite{BB1988, DaiHage2006, Serafino2018, Friedlander1999, Dai2006}, see also  \cite{BonettiniPrato2015, DaiFletcher2005, DaiFletcher2006, Polyak_Levitin1966}. More details  about   selecting  step sizes $\alpha_k$  and matrices $D_k$  can be found in the recent  review  \cite{bonettini2019recent} and  references therein.

Below, we present some  particular instances  of the parameter   $\delta_k\geq 0$ and  the non-monotonicity tolerance parameter $ \nu_ {k} \geq 0$  in Step~3.

\begin{enumerate}
\item {\it Armijo line search} 

	Taking  $\nu_k\equiv 0$, the line search   \eqref{eq:TkArm}  is the well-known (monotone) Armijo line search, see \cite[Section 2.3]{Bertsekas1999}. In this case, we  can take  $\delta_k\equiv 1$ in Step~3. 
	
		\item {\it Max-type line search} 
	
		The earliest non-monotone line search strategy  was proposed  in \cite{Grippo1986}. Let $M>0$ be an integer parameter. In an iteration $k$, this strategy requires a step size $\tau_k>0$ satisfying
		\begin{equation}\label{eq:grippo}
		f\big(x^{k}+ \tau_k(w^k - x^{k})\big) \leq \max_{0\leq j\leq m_k}f(x^{k-j}) + \sigma \tau_k\big\langle \nabla f(x^{k}), w^k - x^{k} \big\rangle,
	\end{equation}
	where $m_0=0$ and $0\leq m_k\leq \min\{m_{k-1}+1, M\}$.  To simplify the notations,  we define $f(x^{\ell(k)}):=\max_{0\leq j\leq m_k}f(x^{k-j})$.  In order to identify \eqref{eq:grippo} as a particular instance of \eqref{eq:TkArm}, we  set
	\begin{equation} \label{eq:casg}
		\nu_{k}= f(x^{\ell(k)})-f(x^k), \quad 0=\delta_{\min}\leq \delta_{k+1}\leq  [f(x^{\ell(k)})- f(x^{\ell(k+1)})]/[f(x^{\ell(k)})-f(x^{k+1})].
	\end{equation}
	Parameters $\nu_{k}$ and $\delta_{k+1}$ in \eqref{eq:casg} satisfy the corresponding conditions in Algorithm~\ref{Alg:GeneralSeach}, i.e.,  $\nu_{k} \geq 0$ and  $\delta_{k+1}\in [\delta_{\min}, 1]$ (with   $\delta_{\min}=0$)  satisfy \eqref{eq:nuk}.  In fact, the definition of $f(x^{\ell(k)})$ implies that   $ f(x^{k})\leq f(x^{\ell(k)})$ and hence $\nu_{k} \geq 0$.  Due to  $\langle \nabla f(x^{k}), w^k - x^{k} \rangle<0$,   it follows from  \eqref{eq:TkArm} that $f(x^{\ell(k)})-f(x^{k+1})>0$. Since   $m_{k+1}\leq m_{k}+1$, we conclude that  $f(x^{\ell(k)})-f(x^{\ell(k+1)}) \geq 0$.  Hence, owing to $ f(x^{k+1})\leq f(x^{\ell(k+1)})$, we obtain $\delta_{k+1}\in [0, 1]$.  Moreover,  \eqref{eq:nuk} is equivalent  to  $\delta_{k+1}[f(x^{k})+\nu_{k}-f(x^{k+1})] \leq(f(x^{k})+\nu_{k}) -  (f(x^{k+1})+ \nu_{k+1})$,  which in turn, taking into account  that $\nu_{k}= f(x^{\ell(k)})-f(x^k)$, is equivalent to second inequality in \eqref{eq:casg}. Thus, \eqref{eq:grippo} is a particular instance of \eqref{eq:TkArm} with  $\nu_{k}$ and $\delta_{k+1}$ defined in \eqref{eq:casg}.  Therefore,  Algorithm~\ref{Alg:GeneralSeach} has as a particular instance the  inexact   projected  version of the scaled gradient method employing   the non-monotone line search  \eqref{eq:grippo}. This version has been considered in \cite{BirginMartinezRaydan2003}; see also  \cite{Bonettini2009, WangLiu2005}.

	\item {\it Average-type line search} 
	
		Let us first recall the definition of the sequence of ``cost updates" $(c_k)_{k\in\mathbb{N}}$  that  characterize the non-monotonous line search proposed in  \cite{ZhangHager2004}. Let   $0\leq \eta_{\min}\leq \eta_{\max}<1$,   $c_0 = f(x_0)$ and  $q_0 = 1$. Choose $\eta_k\in [\eta_{\min},  \eta_{\max}]$ and set
	\begin{equation} \label{eq:zhs}
		q_{k+1}=\eta_kq_{k}+1, \qquad c_{k+1} = [\eta_kq_kc_k + f(x^{k+1})]/q_{k+1}, \qquad \forall k \in \mathbb{N}.
	\end{equation}
	Some algebraic manipulations show that the sequence defined in   \eqref{eq:zhs} is equivalent to
	\begin{equation} \label{eq:zhsn}
		c_{k+1} = (1-1/q_{k+1})c_{k}+f(x^{k+1})/q_{k+1}, \qquad \forall k \in \mathbb{N}.
	\end{equation}
	Since \eqref{eq:nuk} is equivalent  to  $ f(x^{k+1})+ \nu_{k+1}\leq (1-\delta_{k+1})(f(x^{k})+\nu_{k})+\delta_{k+1}f(x^{k+1})$,  it follows from \eqref{eq:zhsn} that  letting  $\nu_{k}=c_k-f(x^k)$ and $\delta_{k+1}=1/q_{k+1}$, Algorithm~\ref{Alg:GeneralSeach} becomes the  inexact   projected  version of the scaled gradient method employing   the non-monotone line search proposed in   \cite{ZhangHager2004}.  Finally,  considering that $q_0 = 1$ and  $\eta_{\max}<1$, the  first equality in   \eqref{eq:zhs} implies  that $q_{k+1}=1+\sum_{j=0}^{k}\prod_{i=0}^{j}\eta_{k-i}\leq \sum_{j=0}^{+\infty} \eta_{\max}^{j}=1/(1-\eta_{\max})$. In this case, due to $\delta_{k+1}=1/q_{k+1}$, we can take   $\delta_{\min}=1-\eta_{\max}>0$ in  Step~3.  For gradient projection methods employing   the non-monotone Average-type line search see, for example, \cite{Paulo2007,Schuverdt2019,  Xihong2018}.

	\end{enumerate}

\begin{remark} \label{rem:outras}
	The general line search in Step~2 of Algorithm~\ref{Alg:GeneralSeach} with  parameters  $\delta_{k+1}$  and  $\nu_{k}$ properly chosen in Step~3, also contains as particular cases the non-monotonous line searches  that appeared in  \cite{Ahookhosh2012,MoLiuYan2007}, see also \cite{GrapigliaSachs2017}.
\end{remark}


\section{Partial asymptotic convergence analysis} \label{Sec:PartialConvRes}
The goal  of this section is to present a partial  convergence result for  the sequence $(x^k)_{k\in\mathbb{N}}$ generated by Algorithm~\ref{Alg:GeneralSeach}, namely, we will prove that every cluster point of $(x^k)_{k\in\mathbb{N}}$ is stationary for problem~\eqref{eq:OptP}.  For that, we state a result that is contained in the proof of \cite[Theorem 4]{GrapigliaSachs2017}. 
\begin{lemma} \label{le:fkvk}
	There holds  $0\leq \delta_{k+1}\big[ f(x^{k})+\nu_{k}-  f(x^{k+1})\big] \leq \big( f(x^{k})+\nu_{k}\big) - \big( f(x^{k+1})+\nu_{k+1}\big)$, for all $k \in \mathbb{N}$. As consequence the sequence   $\left(f(x^k)+\nu_k\right)_{k\in\mathbb{N}}$ is    non-increasing.
	
\end{lemma}

%

Next, we present our first convergence result. It is worth noting that, just as in the classical projected gradient method, we do not need to assume that $f$ has a bounded sub-level  set.

\begin{proposition} \label{pr:statArm}
	Assume that $\lim_{k\to +\infty} \nu_{k} = 0$.   Then, Algorithm~\ref{Alg:GeneralSeach} stops in a finite number of iterations at a stationary point of problem \eqref{eq:OptP}, or generates an infinite sequence $(x^k)_{k\in\mathbb{N}}$ for which every cluster point is stationary for problem~\eqref{eq:OptP}.
\end{proposition}
\begin{proof}
	First, assume that $(x^k)_{k\in\mathbb{N}}$ is finite. In this case, according to Step~1,   there exists $k \in \mathbb{N}$ such that $x^k = w^k \in{\cal P}_{C, \zeta_k}^{D_k}(x^{k}, z^k)$, where $z^k = x^{k}-\alpha_k D_k^{-1}\nabla f(x^{k})$, $0 <{\bar \zeta}<\zeta_k \leq 1$ and $\alpha_k > 0$. Therefore, applying Lemma~\ref{Le:ProjProperty}{\it (ii)} with $x = x^{k}$, $\alpha = \alpha_k$ and $\zeta= \zeta_k$, we conclude that $x^k$ is stationary for problem~\eqref{eq:OptP}.  Now, assume that $(x^k)_{k\in\mathbb{N}}$ is infinite.   Let ${\bar x}$ be a cluster point of $(x^k)_{k\in\mathbb{N}}$ and $(x^{k_j})_{j\in\mathbb{N}}$ be a subsequence of $(x^k)_{k\in\mathbb{N}}$ such that $\lim_{j\to +\infty} x^{k_j} = \bar{x}$. Since $C$ is closed and  $(x^k)_{k\in\mathbb{N}}\subset C$,  we have $\bar{x} \in C$. Moreover, owing to  $\lim_{k\to +\infty} \nu_{k} = 0$, we have $\lim_{j\to +\infty}\left(f(x^{k_j}) +\nu_{k_j}\right) =f(\bar{x})$. Hence, considering that  $\lim_{k\to +\infty} \nu_{k} = 0$ and Lemma~\ref{le:fkvk} implies  that   $\left(f(x^k)+\nu_{k}\right)_{k\in\mathbb{N}}$  is  non-increasing, we conclude that $\lim_{k\to +\infty} f(x^{k})= \lim_{k\to +\infty}\left(f(x^{k}) +\nu_{k}\right) =f(\bar{x})$.
	On the other hand,  due to  $w^k \in {\cal P}_{C, \zeta_k}^{D_k}(x^{k}, z^k)$, where $z^k = x^{k}-\alpha_k \nabla f(x^{k})$,  Definition~\ref{def:InexactM} implies
	\begin{equation} \label{eq:bsw}
		\|w^{k_j} - z^{k_j}\|_{D_k}^2\leq \zeta_{k_j} \| {\cal P}_{C}^{ D_k}(z^{k_j})-z^{k_j}\|_{D_k}^2+(1-\zeta_{k_j})\|x^{k_j}-z^{k_j}\|_{D_k}^2 .
	\end{equation}
	Considering that $(\alpha_k)_{k\in\mathbb{N}}$ and $(\zeta_k)_{k\in\mathbb{N}}$ are bounded, $(D_k)_{k\in\mathbb{N}}\subset  {\cal D}_{\mu}$,  $(x^{k_j})_{j\in\mathbb{N}}$ converges to ${\bar x}$ and $\nabla f$ is continuous, the last inequality together Remark~\ref{re:cproj} and \eqref{eq:pnv}  imply that $(w^{k_j})_{j\in\mathbb{N}}\subset C$ is also bounded. Thus, we can assume without loss of generality that $\lim_{j\to +\infty} w^{k_j} = \bar{w}\in C$.  In addition,  taking into account that  $x^k \neq w^k$ for all $k = 0,1, \ldots$, applying Lemma~\ref{Le:ProjProperty}{\it (i)} with $x = x^{k}$, $\alpha = \alpha_k$, $z(\alpha)=z^k$ and $\zeta= \zeta_k$, we obtain  that $\langle \nabla f(x^k), w^k- x^k \rangle < 0$, for all $k = 0, 1, \ldots$. Therefore,  \eqref{eq:TkArm} and \eqref{eq:IterArm} imply that
	\begin{equation}\label{eq:fmotArmf}
		0 < -\sigma\tau_{k} \big\langle \nabla f(x^{k}), w^{k}-x^{k} \big\rangle \leq f(x^{k}) +\nu_k- f(x^{k+1}), \qquad \forall ~k \in \mathbb{N}.
	\end{equation}
	Now, due $\tau_k \in (0,1]$, for all $k=0,1, \ldots$, we can also assume without loss of generality that $\lim_{j \to +\infty} \tau_{k_j} = \bar{\tau} \in [0,1].$
	Therefore, owing to $\lim_{k\to +\infty} f(x^{k}) =f(\bar{x})$ and $\lim_{k\to +\infty} \nu_{k} = 0$, taking limit in \eqref{eq:fmotArmf} along the  subsequences  $(x^{k_j})_{j\in\mathbb{N}}$,  $(w^{k_j})_{j\in\mathbb{N}}$ and $(\tau_{k_j})_{j\in\mathbb{N}}$  yields
	$
		\bar{\tau} \big\langle \nabla f(\bar{x}), \bar{w}- \bar{x} \big\rangle=0.
	$
	We have two possibilities: $\bar{\tau} > 0$ or $\bar{\tau} = 0$. If $\bar{\tau} > 0$, then  $\big\langle \nabla f(\bar{x}), \bar{w}- \bar{x} \big\rangle = 0.$  
		Now, we  assume that $\bar{\tau} = 0$. In this case, for all $j$ large enough, there exists $0<\hat\tau_{k_j}\leq \min\{1,\tau_{k_j}/\underline\omega\}$ such that
	\begin{equation}\label{eq:ffA10}		
		f\big(x^{k_j}+\hat\tau_{k_j} (w^{k_j} - x^{k_j})\big) > f(x^{k_j}) + \sigma \hat\tau_{k_j} \big\langle \nabla f(x^{k_j}), w^{k_j} - x^{k_j} \big\rangle +\nu_{k_j}.
	\end{equation}
On the other hand, by the mean value theorem, there exists $\xi_{k_j}\in(0,1)$ such that
		$$\langle \nabla f\big(x^{k_j}+\xi_{k_j}\hat\tau_{k_j} (w^{k_j} - x^{k_j})\big), \hat\tau_{k_j} (w^{k_j} - x^{k_j})\rangle = f\big(x^{k_j}+\hat\tau_{k_j} (w^{k_j} - x^{k_j})\big) - f(x^{k_j}).$$
Combining this equality with \eqref{eq:ffA10}, and taking into account that $\nu_{k_j}\geq 0$, we have 
		$$\langle \nabla f\big(x^{k_j}+\xi_{k_j}\hat\tau_{k_j} (w^{k_j} - x^{k_j})\big), \hat\tau_{k_j} (w^{k_j} - x^{k_j})\rangle>\sigma \hat\tau_{k_j} \big\langle \nabla f(x^{k_j}), w^{k_j} - x^{k_j} \big\rangle,$$
 for $j$ large enough. Since $0<\hat\tau_{k_j}\leq \min\{1,\tau_{k_j}/\underline\omega\}$, it follows that $\lim_{j\to\infty} \hat\tau_{k_j} \|w^{k_j} - x^{k_j}\|=0$. Then, dividing both sides of the above inequality by $\hat\tau_{k_j}>0$ and taking limits as $j$ goes to $+\infty$, we conclude that $ \langle \nabla f(\bar{x}), \bar{w}-\bar{x} \rangle \geq \sigma \langle \nabla f(\bar{x}), \bar{w}-\bar{x} \rangle$.   Hence, due to $\sigma \in (0, 1)$, we obtain $\langle \nabla f(\bar{x}), \bar{w}-\bar{x} \rangle \geq 0$. We recall that $\langle \nabla f(x^{k_j}), w^{k_j}- x^{k_j} \rangle < 0$, for all $j=0, 1, \ldots$, which taking limit as $j$ goes to $+\infty$ yields $\langle \nabla f(\bar{x}), \bar{w}-\bar{x} \rangle \leq 0$. Hence, we also have $\langle \nabla f(\bar{x}), \bar{w}-\bar{x} \rangle = 0$. Therefore, for any of the two possibilities, $\bar{\tau} > 0$ or $\bar{\tau} = 0$, we have $\langle \nabla f(\bar{x}), \bar{w}-\bar{x} \rangle = 0$. On the other hand,  since  $(\alpha_k)_{k\in\mathbb{N}}$ and    $(\zeta_k)_{k\in\mathbb{N}}$   are  bounded, we also assume without loss of generality that $\lim_{j \to +\infty} \alpha_{k_j} = \bar{\alpha} \in [\alpha_{\min}, \alpha_{\max}]$ and $\lim_{j \to +\infty} \zeta_{k_j} = {\bar \zeta} \in [\zeta_{\min}, 1]$. Thus, since Remark~\ref{re:cproj} implies that
	$$
		\lim_{j \to +\infty}{\cal P}_{C}^{D_{k_j}}(z^{k_j})= {\cal P}_{C}^{\bar D}(\bar{z}),
	$$
	and considering that $\lim_{j\to +\infty} x^{k_j} = \bar{x}\in C$, $\lim_{j\to +\infty} w^{k_j} = \bar{w}\in C$, $\lim_{j \to +\infty} \tau_{k_j} = \bar{\tau} \in [0,1]$,   $\lim_{j \to +\infty} D_{k_j} = \bar{D}\in {\cal D}_{\mu}$, taking limit in \eqref{eq:bsw},  we conclude that
	$$
		\|\bar{w} -\bar{z}\|_{\bar D}^2\leq  {\bar \zeta}  \| {\cal P}_{C}^{\bar D}(\bar{z})-\bar{z}\|_{\bar D}^2+(1- {\bar \zeta} )\| \bar{x}-\bar{z}\|_{\bar D}^2 ,
	$$
	where $\bar{z} = \bar{x}-{\bar \alpha} \nabla f(\bar{x})$. Hence, Definition~\ref{def:InexactM} implies  that ${\bar w}\in  {\cal P}_{C,{\bar \zeta}}^{\bar D}( {\bar x}, {\bar z})$, where $\bar{z} = \bar{x}-{\bar \alpha} \nabla f(\bar{x})$. Therefore, due to $\langle \nabla f(\bar{x}), \bar{w}-\bar{x} \rangle = 0$, we can apply second sentence in Lemma \ref{Le:ProjProperty}{\it (iii)} with $x = \bar{x}$, $z({\bar \alpha}) = \bar{z}$ and $w({\bar \alpha}) = \bar{w}$, to conclude that $\bar{x}$ is stationary for problem~\eqref{eq:OptP}.
\end{proof}

The tolerance parameter $\nu_{k}$ that controls the non-monotonicity of the line search must be smaller and smaller as the sequence $(x^k)_{k\in\mathbb{N}}$  tends to  a stationary point. Next corollary presents a general condition for this property, its proof can be found in \cite[Theorem 4]{GrapigliaSachs2017}.
\begin{corollary} \label{cr:fkvk}
	If $\delta_{\min}>0$,  then  $\sum_{k=0}^{+\infty} \nu_k<+\infty$. Consequently, $\lim_{k\to +\infty} \nu_{k} = 0$.
\end{corollary}

The Armijo and the non-monotone Average-type line searches discussed in Section~\ref{Sec:SGM} satisfy  the assumption of Corollary~\ref{cr:fkvk}, i.e., $\delta_{\min}>0$.  However,    for  the  non-monotone Max-type line search,  we   can only guarantee that $\delta_{\min}\geq 0$. Hence,  we can not apply  Corollary~\ref{cr:fkvk}  to conclude that $\lim_{k\to +\infty} \nu_{k}~=~0$.  In the next proposition, we will deal with this case separately.

\begin{proposition} \label{pr;gripponuo}
	Assume that the sequence  $(x^k)_{k\in\mathbb{N}}$ is generated by Algorithm~\ref{Alg:GeneralSeach} with the  non-monotone line  search \eqref{eq:grippo}, i.e.,  $\nu_{k}= f(x^{\ell(k)})-f(x^k)$ for all  $k \in \mathbb{N}$. In addition,  assume that the level set $C_{0}:=\{ x\in C: ~ f(x)\leq f(x^0) \}$ is bounded and $\nu_0= 0$.  Then, $\lim_{k\to +\infty} \nu_{k} = 0$.
\end{proposition}
\begin{proof}
	First of all, note that     $w^k \in   {\cal P}_{C,\zeta_k}^{D_k}(x^{k}, z^k)$,  where $z^k = x^{k}-\alpha_k D^{-1} _k\nabla f(x^{k})$ and $D_k\in {\cal D}_{\mu}$. Thus,  applying Lemma~\ref{Le:ProjProperty}{\it (i)}  with $x=x^k$, $w(\alpha) = w^k$, $z = z^k$ and $\zeta= \zeta_k$, we obtain
	\begin{equation}\label{eq:apna}
		\|w^k-x^k\|^2\leq -2\mu \alpha_{\max}\langle \nabla f(x^{k}), w^k-x^{k}\rangle, \qquad \forall k \in \mathbb{N}.
	\end{equation}
	On the other hand, due to   $f(x^{\ell(k)})= f(x^k)+ \nu_{k}$,  Lemma~\ref{le:fkvk} implies that    $(f(x^{\ell(k)}))_{k\in\mathbb{N}}$ is    non-increasing and $  f(x^{k+1})\leq   f(x^{k+1})+\nu_{k+1}\leq f(x^{k})+\nu_{k}\leq  f(x^{0})$. Hence, we have  $(x^k)_{k\in\mathbb{N}}\subset C_{0}$ and, as a consequence,    $(f(x^{\ell(k)}))_{k\in\mathbb{N}}$ converges. Note that $\ell(k)$   is an integer such that 
	\begin{equation}\label{eq:lk}
		k-m_k\leq \ell(k)\leq k.
	\end{equation}
	Since $x^{{\ell(k)}}=x^{{\ell(k)}-1}+ \tau_{{\ell(k)}-1} (w^{{\ell(k)}-1} - x^{{\ell(k)} -1})$,  \eqref{eq:grippo}  implies that
	$$
		f\big(x^{\ell(k)}\big)  \leq f\big(x^{\ell({{\ell(k)}-1})}\big)+ \sigma \tau_{{\ell(k)}-1}\big\langle \nabla f(x^{{\ell(k)}-1}), w^{{\ell(k)}-1} - x^{{\ell(k)}-1} \big\rangle,
	$$
	for all $k>M$.  In view of   $(f(x^{\ell(k)}))_{k\in\mathbb{N}}$  be convergent, $\langle \nabla f(x^{k}), w^k - x^{k} \rangle<0$ for all  $k \in \mathbb{N}$, and taking into account that   $\tau_k  \in (0, 1]$,  the last inequality together \eqref{eq:apna} implies that
	\begin{equation}\label{eq:apcss}
		\lim_{k\to +\infty} \tau_{{\ell(k)}-1}\|w^{{\ell(k)}-1}-x^{{\ell(k)}-1}\|=0.
	\end{equation}
	We proceed to  prove that  $\lim_{k\to +\infty} f(x^{k})= \lim_{k\to +\infty} f(x^{\ell(k)})$. For that, set ${\hat \ell}(k):=\ell(k+M+2)$. First, we prove by induction that, for all  $j\geq 1$,  the following two equalities  hold
	\begin{equation}\label{eq:ind}
		\lim_{k\to +\infty}  \tau_{{\hat \ell}(k)-j}\|w^{{{\hat \ell}(k)}-j}-x^{{{\hat \ell}(k)}-j}\|=0, \qquad \lim_{k\to +\infty} f(x^{{\hat \ell}(k)-j})= \lim_{k\to +\infty} f(x^{\ell(k)}),
	\end{equation}
	where we are  considering $k\geq j-1$. Assume that $j=1$. Since  $\{{\hat \ell}(k): ~k\in\mathbb{N}\}\subset \{{\ell}(k): ~k\in\mathbb{N}\}$, the first equality in \eqref{eq:ind} follows from \eqref{eq:apcss}. Hence, $\lim_{k\to +\infty} \|x^{{{\hat \ell}(k)}}-x^{{\hat \ell(k)}-1}\|=0$. Since  $C_{0}$ is  compact and  $f$ is uniformly continuous on $C_{0}$, we have $  \lim_{k\to +\infty} f(x^{{\hat \ell}(k)-1})=\lim_{k\to +\infty} f(x^{{\hat \ell(k)}})$, which again using that $\{{\hat \ell}(k): ~k\in\mathbb{N}\}\subset \{{\ell}(k): ~k\in\mathbb{N}\}$ implies the second equality in \eqref{eq:ind}. Assume that \eqref{eq:ind} holds for $j$. Again, due to  $x^{{{\hat \ell}(k)}-j}=x^{{{{\hat \ell}(k)}-j}-1}+ \tau_{{{{\hat \ell}(k)}-j}-1} (w^{{{{\hat \ell}(k)}-j}-1} - x^{{{{\hat \ell}(k)}-j} -1})$,  \eqref{eq:grippo}  implies that
	$$
		f\big(x^{{{\hat \ell}(k)}-j}\big)  \leq f\big(x^{\ell({{{{\hat \ell}(k)}j}-(j+1)})}\big)+ \sigma \tau_{{{{\hat \ell}(k)}}-(j+1)}\big\langle \nabla f(x^{{{{\hat \ell}(k)}}-(j+1)}), w^{{{{\hat \ell}(k)}}-(j+1)} - x^{{{{\hat \ell}(k)}}-(j+1)} \big\rangle.
	$$
	Similar argument used to obtain \eqref{eq:apcss} yields  $\lim_{k\to +\infty} \tau_{{{{\hat \ell}(k)}}-(j+1)}\|w^{{{{\hat \ell}(k)}}-(j+1)}-x^{{{{\hat \ell}(k)}}-(j+1)}\|=0$. Thus,  the first equality in \eqref{eq:ind} holds for $j+1$, which  implies  $\lim_{k\to +\infty} \|x^{{{\hat \ell}(k)}-j}-x^{{{{\hat \ell}(k)}}-(1+j)}\|=0$.  Again, the   uniformly continuity of $f$ on $C_{0}$ gives
	$$
		\lim_{k\to +\infty} f(x^{{\hat \ell}(k)-(j+1)})=\lim_{k\to +\infty} f(x^{{\hat \ell}(k)-j}),
	$$
	which shows that the second equality in \eqref{eq:ind} holds for $j+1$. From  \eqref{eq:lk}  and ${\hat \ell}(k):=\ell(k+M+2)$, we obtain ${\hat \ell}(k)-k-1\leq M+1$. Thus,  taking into account that
	$$
		x^{k+1}=x^{{\hat \ell}(k)}- \sum_{j=1}^{{\hat \ell}(k)-k-1} \tau_{{\hat \ell}(k)-j} \big(w^{{\hat \ell}(k)-j} - x^{{\hat \ell}(k)-j}\big),
	$$
	it follows from the first inequality in \eqref{eq:ind} that $ \lim_{k\to +\infty} \|x^{k+1}-x^{{\hat \ell}(k)}\|=0$. Hence, due to  $f$ be uniformly continuous on $C_{0}$ and $(f(x^{\ell(k)}))_{k\in\mathbb{N}}$  be convergent,  we conclude that
	$$ \lim_{k\to +\infty} f(x^{k})=\lim_{k\to +\infty} f(x^{{\hat \ell}(k)})= \lim_{k\to +\infty} f(x^{\ell(k)}),$$
	and  considering that $\nu_{k}= f(x^{\ell(k)})-f(x^k)$ the desired results follows.
\end{proof}
\begin{remark}
	Let  $C_{0}:=\{ x\in C: ~ f(x)\leq f(x^0) \}$ be  bounded and    $(x^k)_{k\in\mathbb{N}}$ be  generated by Algorithm~\ref{Alg:GeneralSeach} with the  non-monotone line  search \eqref{eq:grippo} with $\nu_0= 0$.     Then, combining Propositions~\ref{pr:statArm} and \ref{pr;gripponuo}, we conclude that  $(x^k)_{k\in\mathbb{N}}$  is either finite terminating at a stationary point of problem~\eqref{eq:OptP}, or infinite,  and every cluster point of $(x^k)_{k\in\mathbb{N}}$ is stationary for problem~\eqref{eq:OptP}.  Therefore, we have an alternative proof for the result obtained in \cite[Theorem 2.1]{BirginMartinezRaydan2003}. 
	 \end{remark}
Due to Proposition~\ref{pr:statArm}, {\it from now on we assume that the sequence $(x^k)_{k\in\mathbb{N}}$ generated by Algorithm~\ref{Alg:GeneralSeach} is infinite}.


\section{Full asymptotic convergence  and complexity  analysis } \label{Sec:FullConvRes}

The purpose of this section is twofold.  We will first prove, under suitable assumptions, the full convergence of the  sequence $(x^k)_{k\in\mathbb{N}}$  and then we will present  iteration-complexity bounds for it.   For this end,  we need to be more restrictive both with respect to the inexact projection in \eqref{eq:PInexArm} and in the tolerance parameter that controls the non-monotonicity of the line search used in \eqref{eq:TkArm}. More precisely,  we   assume  that in Step~1  of  Algorithm~\ref{Alg:GeneralSeach}:
\begin{itemize}
	\item[{\bf A1.}] For all $k \in \mathbb{N}$, we take   $w^k \in   {\cal R}_{C,\gamma_k}^{D_k}(x^{k}, z^k)$   with $\gamma_k=(1-\zeta_k)/2$.
\end{itemize}
It is worth recalling  that, taking   the parameter $\gamma_k=(1-\zeta_k)/2 $, it follows from Lemma~\ref{pr:condrip} that  $ {\cal R}_{C,\gamma_k}^{D_k}(x^{k}, z^k) \subset {\cal P}_{C, \zeta_k}^{D_k}(x^{k}, z^k)$. In addition, we also assume that   in Step~2 of  Algorithm~\ref{Alg:GeneralSeach}:

\begin{itemize}
	\item[{\bf A2.}]For all $k \in \mathbb{N}$,  we take  $0\leq \nu_{k}$ such that  $\sum_{k=0}^{+\infty} \nu_k<+\infty$.
\end{itemize}
It follows from Corollary~\ref{cr:fkvk} that the Armijo and the non-monotone Average-type line searches discussed in Section~\ref{Sec:SGM} satisfy  Assumption {\bf A2}.


\subsection{Full asymptotic convergence analysis} \label{SubSec:CAnalysisAF}

In this section, we  prove  the full convergence of the  sequence  $(x^k)_{k\in\mathbb{N}}$ satisfying {\bf A1} and {\bf A2}.  We will begin establishing  a basic inequality for    $(x^k)_{k\in\mathbb{N}}$.  To simplify  notations, we define the constant
\begin{equation} \label{eq:eta}
	\xi := \dfrac{2 \alpha_{\max}}{\sigma} > 0.
\end{equation}

\begin{lemma}\label{Le:xkArm}
	For each  $x\in C$, there holds
	\begin{equation}\label{eq:xkArm}
		\|x^{k+1}-x\|_{D_k}^2 \leq \|x^k-x\|_{D_k}^2 + 2\alpha_k\tau_k \big\langle \nabla f(x^k), x-x^k\big\rangle + \xi \big[f(x^k) - f(x^{k+1})+ \nu_k \big], \quad \forall ~k \in \mathbb{N}.
	\end{equation}
\end{lemma}

\begin{proof}
	We know that $\|x^{k+1}-x\|_{D_k}^2 = \|x^k-x\|_{D_k}^2 + \|x^{k+1}-x^k\|_{D_k}^2 - 2 \langle {D_k} ( x^{k+1}-x^k), x-x^k \rangle$, for all $x \in C$ and $k \in \mathbb{N}$. Thus, using \eqref{eq:IterArm}, we have
	\begin{equation}\label{eq:xkArm1}
		\|x^{k+1}-x\|_{D_k}^2 = \|x^k-x\|_{D_k}^2 + \tau_k^2\|w^k - x^{k}\|_{D_k}^2 - 2 \tau_k \big\langle {D_k}(w^k - x^{k}), x-x^k \big\rangle, \qquad \forall ~k \in \mathbb{N}.
	\end{equation}
	On the other hand, since  $w^k \in{\cal R}_{C,\gamma_k}^{D_k}(x^{k}, z^k)$ with $z^k = x^{k}-\alpha_k D_k^{-1} \nabla f(x^{k})$, it follows from Definition~\ref{def:InexactProjC},  with $y=x$, $D = D_k$, $u = x^k$, $v = z^k$, $w = w^k$,  and $\gamma = \gamma_k$,  that
	$$
		\big\langle D_k(x^k-\alpha_kD_k^{-1}\nabla f(x^k)-w^k), x-w^k\big\rangle \leq \gamma_k \|w^k - x^{k}\|_{D_k}^2, \qquad \forall ~k \in \mathbb{N}.
	$$
	Hence,  after some algebraic manipulations in the last inequality, we have
	$$
		-\big\langle D_k(w^k-x^k), x-x^k\big\rangle \leq \alpha_k \big\langle \nabla f(x^k), x-w^k \big\rangle - (1-\gamma_k) \|w^k-x^k\|_{D_k}^2.
	$$
	Combining the last inequality with \eqref{eq:xkArm1},  we conclude  that
	\begin{equation} \label{eq:xkArm3}
		\|x^{k+1}-x\|_{D_k}^2 \leq \|x^k-x\|_{D_k}^2 - \tau_k \big[2(1-\gamma_k) - \tau_k \big] \|w^k-x^k\|_{D_k}^2 + 2\tau_k\alpha_k \big\langle \nabla f(x^k), x-w^k\big\rangle.
	\end{equation}
	Since $0 \leq \gamma_k <(1-{\zeta_{\min}})/2 < 1/2$ and $\tau_k \in (0, 1]$, we have $2(1-\gamma_k) - \tau_k > {\zeta_{\min}} > 0$. Thus, it follows from \eqref{eq:xkArm3} that
	$$
		\|x^{k+1}-x\|_{D_k}^2 \leq \|x^k-x\|_{D_k}^2 + 2\tau_k\alpha_k \big\langle \nabla f(x^k), x-w^k\big\rangle, \qquad \forall ~k \in \mathbb{N}.
	$$
	Thus, considering that $\big\langle \nabla f(x^k), x-w^k\big\rangle = \big\langle \nabla f(x^k), x-x^k\big\rangle + \big\langle \nabla f(x^k), x^k-w^k \big\rangle$ and taking into account \eqref{eq:TkArm}, we conclude that
	\begin{equation} \label{eq;ali}
		\|x^{k+1}-x\|_{D_k}^2 \leq  \|x^k-x\|_{D_k}^2 + 2\tau_k\alpha_k \left\langle \nabla f(x^k),x-x^k\right\rangle + \frac{2 \alpha_k}{\sigma} \big[f(x^k)-f(x^{k+1})+\nu_k\big],
	\end{equation}
	for all $ k \in \mathbb{N}$.  On the other hand, applying Lemma~\ref{Le:ProjProperty}{\it (iii)}  with $x=x^k$, $\alpha=\alpha_k$, $D = D_k$, $w(\alpha) = w^k$, $z = z^k$ and $\zeta= \zeta_k$, we obtain  $\langle \nabla f(x^k), w^k- x^k \rangle <  0$, for all  $ k \in \mathbb{N}$. Therefore, it follows from \eqref{eq:TkArm} and \eqref{eq:IterArm} that $0 < -\sigma\tau_{k} \big\langle \nabla f(x^{k}), w^{k}-x^{k} \big\rangle \leq f(x^{k}) - f(x^{k+1})+\nu_k$, to all $k \in \mathbb{N}$. Hence, due to $0< \alpha_k\leq  \alpha_{\max}$,  we have
	$$
		\alpha_k[f(x^k)-f(x^{k+1})+\nu_k] < \alpha_{\max} [f(x^k)-f(x^{k+1})+\nu_k], \qquad \forall k \in \mathbb{N}.
	$$
	Therefore,  \eqref{eq:xkArm} follows from  the combination of the  last inequality with \eqref{eq:eta} and  \eqref{eq;ali}.
\end{proof}

For proceeding with the analysis of  the behavior of the sequence $(x^k)_{k\in\mathbb{N}}$,  we define the following auxiliary set
\begin{equation*}\label{eq:SetTArm}
	U := \left\{x \in C: f(x) \leq \inf_{k\in {\mathbb N}}\left(f(x^{k})+\nu_k\right) \right\}.
\end{equation*}

\begin{corollary} \label{cor:xkquasifeArm}
	Assume that $f$ is a convex function. If $U \neq \varnothing$, then $(x^k)_{k\in\mathbb{N}}$ converges to a stationary point of problem~\eqref{eq:OptP}.
\end{corollary}
\begin{proof}
	Let $x \in U$.  Since  $f$ is convex, we have $0\geq f(x)-(f(x^k)+\nu_k)\geq \langle \nabla f(x^k),x-x^k\rangle -\nu_k$, for all $k\in \mathbb{N}$. Thus, $\langle \nabla f(x^k),x-x^k\rangle\leq \nu_k$, for all $k\in \mathbb{N}$.   Using Lemma \ref{Le:xkArm} and taking into account  that  $\tau_k \in (0, 1]$  and  $0<\alpha_{\min}\leq \alpha_k \leq \alpha_{\max}$,  we obtain
	$$
		\|x^{k+1}-x\|_{D_k}^2 \leq \|x^k-x\|_{D_k}^2+2 \alpha_{\max} \nu_k+ \xi \big[f(x^k) - f(x^{k+1}) +\nu_k\big], \quad \forall~k \in \mathbb{N}.
	$$
	Defining $\epsilon_k =  2 \alpha_{\max} \nu_k + \xi \big[f(x^k) - f(x^{k+1}) +\nu_k\big]$, we have $\|x^{k+1}-x\|_{D_k}^2 \leq \|x^k-x\|_{D_k}^2+ \epsilon_k$, for all $k \in \mathbb{N}$. On the other hand, summing $\epsilon_k$ with $k = 0, 1, \ldots, N$ and using  Corollary~\ref{cr:fkvk},  we have
	$$
		\sum_{k=0}^N \epsilon_k \leq 2   \alpha_{\max} \sum_{k=0}^N \nu_k +  \xi \left(f(x^0) - f(x) + \sum_{k=0}^{N+1} \nu_k \right) < +\infty, \qquad \forall N \in \mathbb{N}.
	$$
	Hence, $\sum_{k=0}^{+\infty} \epsilon_k<+\infty$.  Thus, it follows from  Definition~\ref{def:QuasiFejer}  that $(x^k)_{k\in\mathbb{N}}$ is quasi-Fej\'er convergent   to $U$ with respect to the sequence  $(D_k)_{k\in\mathbb{N}}$ . Since  $U$ is nonempty, it follows from Theorem \ref{teo.qf} that $(x^k)_{k\in\mathbb{N}}$ is bounded, and therefore it has cluster points. Let $\bar{x}$ be a cluster point of $(x^k)_{k\in\mathbb{N}}$ and $(x^{k_j})_{j\in\mathbb{N}}$ be a subsequence of $(x^k)_{k\in\mathbb{N}}$ such that $\lim_{j \to \infty} x^{k_j} = \bar{x}$. Considering that $f$ is continuous and $\lim_{k\to +\infty} \nu_{k} = 0$, we have $\lim_{j \to \infty} (f(x^{k_j})+\nu_{k_j})= f(\bar{x})$.  On the other hand, Lemma~\ref{le:fkvk} implies that  $\left(f(x^k)+\nu_k\right)_{k\in\mathbb{N}}$ is  non-increasing. Thus  $\inf_{k\in {\mathbb N}}(f(x^{k})+\nu_k)= \lim_{k \to \infty} (f(x^{k})+\nu_{k}) = f(\bar{x}).$ Hence, $\bar{x} \in U$, and  Theorem~\ref{teo.qf}  implies that $(x^k)_{k\in\mathbb{N}}$ converges to $\bar{x}$.  The conclusion is obtained  by  using   Proposition ~\ref{pr:statArm}.
\end{proof}

\begin{theorem}
	If $f$ is a convex function and $(x^k)_{k\in\mathbb{N}}$ has no cluster points,  then $\Omega^* = \varnothing$, $\lim_{k \to \infty} \|x^k\|= +\infty$, and $\inf_{k\in {\mathbb N}} f(x^k) = \inf \{f(x) : x \in C\}$.
\end{theorem}
\begin{proof}
	Since $(x^k)_{k\in\mathbb{N}}$ has no cluster points, then $\lim_{k \to \infty} \|x^k\|= +\infty$. Assume by contradiction that $\Omega^* \neq  \varnothing$.  Thus, there exists  $\tilde{x}\in C$, such that  $f(\tilde{x}) \leq f(x^k)$ for all $k\in {\mathbb N}$. Therefore, $\tilde{x} \in U$. Using Corollary \ref{cor:xkquasifeArm}, we obtain that $(x^k)_{k\in\mathbb{N}}$ is convergent, contradicting that $\lim_{k \to \infty} \|x^k\|= \infty$. Therefore, $\Omega^* = \varnothing$. Now, we claim that $\inf_{k\in {\mathbb N}} f(x^k) = \inf \{f(x) : x \in C\}$.   If $\inf_{k\in {\mathbb N}} f(x^k) = -\infty$, the claim holds. Assume by contraction that   $\inf_{k\in {\mathbb N}} f(x^k) >  \inf_{x \in C} f(x)$.  Thus,  there exists $\tilde{x} \in C$ such that $f(\tilde{x}) \leq f(x^k)\leq f(x^k)+\nu_k $,  for all $k\in {\mathbb N}$.  Hence, $U \neq \varnothing$.  Using Corollary \ref{cor:xkquasifeArm}, we have that  $(x^k)_{k\in\mathbb{N}}$ is convergent, contradicting again $\lim_{k \to \infty} \|x^k\|= +\infty$ and concluding the proof.
\end{proof}

\begin{corollary}
	If $f$ is a convex function and $(x^k)_{k\in\mathbb{N}}$ has at least one cluster point, then    $(x^k)_{k\in\mathbb{N}}$ converges to a stationary point of problem~\eqref{eq:OptP}.
\end{corollary}
\begin{proof}
	Let $\bar{x}$ be a cluster point of  the sequence $(x^k)_{k\in\mathbb{N}}$ and $(x^{k_j})_{j\in\mathbb{N}}$ be a subsequence of $(x^k)_{k\in\mathbb{N}}$ such that $\lim_{j\to +\infty} x^{k_j} = \bar{x}$. Considering that  $f$ is continuous and $\lim_{k\to +\infty} \nu_{k} = 0$, we have $\lim_{j \to \infty} (f(x^{k_j})+\nu_{k_j})= f(\bar{x})$.    On the other hand,  Corollary~\ref{cr:fkvk}   implies that   $(f(x^{k})+\nu_k)_{k\in\mathbb{N}}$ is non-increasing.  Hence, we have   $\inf_{k\in {\mathbb N}} (f(x^{k})+\nu_{k})=\lim_{k\to \infty} (f(x^{k})+\nu_{k})= f(\bar{x})$.  Therefore $\bar{x} \in U$. Using  Corollary~\ref{cor:xkquasifeArm}, we obtain that $(x^k)_{k\in\mathbb{N}}$ converges to a stationary point $\tilde{x}\in C$ of  problem \eqref{eq:OptP}.
\end{proof}

\begin{theorem}
	Assume that $f$ is a convex function and  $\Omega^* \neq \varnothing$. Then,   $(x^k)_{k\in\mathbb{N}}$ converge to an optimal solution of problem~\eqref{eq:OptP}.
\end{theorem}
\begin{proof}
	If $\Omega^* \neq \varnothing$, then   $U \neq \varnothing$.   Therefore,  Corollary~\ref{cor:xkquasifeArm} implies  that $(x^k)_{k\in\mathbb{N}}$ converges to a stationary point of  problem~\eqref{eq:OptP} and,  due to  $f$ be   convex, this point  is also an optimal solution.
\end{proof}


\subsection{Iteration-complexity bound}\label{SubSec:IterCompArm}

In the section, we preset some  iteration-complexity bounds related to  the sequence $(x^k)_{k\in\mathbb{N}}$ generated by  Algorithm~\ref{Alg:GeneralSeach}.  For that, besides  assuming  {\bf A1} and {\bf A2},  we also need the following assumption.
\begin{itemize}
	\item[{\bf A3.}] The  gradient $\nabla f$ of $f$ is  Lipschitz continuous with constant $L>0$.
\end{itemize}
For simple notations, we define the  following positive constant

\begin{equation} \label{eq;taumin}
	\tau_{\min} := \min \left\{1, \frac{\underline\omega(1-\sigma)}{{\alpha_{\max}}\mu L}\right\}.
\end{equation}

\begin{lemma}\label{Le:tauminArm}
	The steepsize $\tau_k$ in Algorithm~\ref{Alg:GeneralSeach} satisfies $\tau_k \geq \tau_{\min}$.
\end{lemma}
\begin{proof}
	First, we assume that $\tau_k=1$. In this case, we have $\tau_k \geq \tau_{\min}$ and the required inequality holds. Now, we assume that $\tau_k<1$. Thus, it follows from  \eqref{eq:TkArm}   that there exists $0<\hat\tau_k\leq \min\{1,\tau_k/\underline\omega\}$ such that
	\begin{equation}\label{eq:ffA1}
		f\big(x^{k}+ \hat\tau_k (w^{k} - x^{k})\big) > f(x^{k}) + \sigma \hat\tau_k  \big\langle \nabla f(x^{k}), w^{k} - x^{k} \big\rangle+\nu_k.
	\end{equation}
	Considering that we are under assumption {\bf A3}, we  apply   Lemma \ref{Le:derivlipsch} to  obtain 
	\begin{equation}\label{eq:ffA2}
		f\big(x^{k}+ \hat\tau_k (w^{k} - x^{k})\big) \leq f(x^{k}) + \hat\tau_k \big\langle \nabla f(x^{k}), w^{k} - x^{k} \big\rangle +\frac{L}{2}\hat\tau_k^2 \|w^k-x^k\|^2.
	\end{equation}
	Hence, the combination of \eqref{eq:ffA1} with \eqref{eq:ffA2}  yields
	\begin{equation}\label{eq:ffA3}
		(1-\sigma) \big\langle \nabla f(x^{k}), w^{k} - x^{k} \big\rangle + \frac{L}{2}\hat\tau_k \|w^k-x^k\|^2 >  \frac{\nu_k}{\hat\tau_k} .
	\end{equation}
	On the order hand,    $w^k \in   {\cal R}_{C,\gamma_k}^{D_k}(x^{k}, z^k)$ with    $\gamma_k=(1-\zeta_k)/2$, where $z^k = x^{k}-\alpha_k D^{-1} _k\nabla f(x^{k})$. Thus,  applying Lemma~\ref{Le:ProjProperty}{\it (i)}  with $x=x^k$, $w(\alpha) = w^k$, $z = z^k$ and $\zeta= \zeta_k$, we obtain
	$$
		\big\langle \nabla f(x^{k}), w^k-x^{k}\big\rangle \leq -\frac{1}{2\alpha_k} \|w^k-x^k\|_{D_k}^2.
	$$
	Hence,  considering that    $\frac{1}{\mu}  \|w^k-x^k\|^2 \leq  \|w^k-x^k\|_{D_k}^2$  and $0<\alpha_k \leq \alpha_{\max}$, the last inequality 	implies
	$$
		\big\langle \nabla f(x^{k}), w^k-x^{k}\big\rangle \leq  -\frac{1}{2\alpha_{\max}\mu} \|w^k-x^k\|^2.
	$$
	The combination of  the last inequality with \eqref{eq:ffA3} yields
	$$
		\left(-\frac{(1-\sigma)}{2{\alpha_{\max}}\mu} + \frac{L}{2}\hat\tau_k \right)\|w^k-x^k\|^2>  \frac{\nu_k}{\hat\tau_k} \geq 0 .
	$$
	Thus, since $\hat\tau_k\leq \tau_k/\underline\omega$,   we obtain  $\tau_k\geq \underline\omega \hat\tau_k >\underline\omega(1-\sigma)/({\alpha_{\max}}\mu L)\geq \tau_{\min}$ and the proof is concluded.
\end{proof}

Considering  that $ {\cal R}_{C,\gamma_k}^{D_k}(x^{k}, z^k) \subset {\cal P}_{C, \zeta_k}^{D_k}(x^{k}, z^k)$, it follows from Lemma~\ref{Le:ProjProperty}{\it (ii)}  that if $x^k \in {\cal R}_{C,\gamma_k}^{D_k}(x^{k}, z^k)$, then the point $x^k$ is stationary for problem \eqref{eq:OptP}. Since $w^k \in {\cal R}_{C,\gamma_k}^{D_k}(x^{k}, z^k)$, the quantity $\|w^k-x^k\|$ can be seen as a measure of stationarity of the point $x^k$. In next theorem, we present an iteration-complexity bound for this quantity,  which is a constrained inexact  version of  \cite[Theorem~1]{GrapigliaSachs2017}.

\begin{theorem} \label{eq:theocomp}
	Let $ \tau_{\min}$ be defined in \eqref{eq;taumin}. Then, for every $N \in \mathbb{N}$, the following inequality holds
	$$
		\min\left\{\|w^k-x^k\| :~ k= 0, 1 \ldots, N-1\right\} \leq \sqrt{\frac{2{\alpha_{\max}}\mu\left[ f(x^0)-f^* +\sum_{k= 0}^{\infty}\nu_k\right] }{\sigma \tau_{\min}}} \frac{1}{\sqrt{N}}.
	$$
\end{theorem}

\begin{proof}
	Since  $w^k \in   {\cal R}_{C,\gamma_k}^{D_k}(x^{k}, z^k)$ with    $\gamma_k=(1-\zeta_k)/2$, where $z^k = x^{k}-\alpha_k D^{-1} _k\nabla f(x^{k})$,  applying  Lemma~\ref{Le:ProjProperty}{\it (i)} with $x=x^k$, $w(\alpha) = w^k$, $z = z^k$ and $\zeta= \zeta_k$, and taking into account that $(1/\mu)  \|w^k-x^k\|^2 \leq  \|w^k-x^k\|_{D_k}^2$  and $0<\alpha_k \leq \alpha_{\max}$, we obtain
	\begin{equation*}\label{eq:fD2}
		\big\langle \nabla f(x^{k}), w^k-x^{k}\big\rangle \leq -\frac{1}{2\alpha_k} \|w^k-x^k\|_{D_k}^2\leq -\frac{1}{2\alpha_{\max}\mu} \|w^k-x^k\|^2.
	\end{equation*}
	The definition of $\tau_k$  and \eqref{eq:TkArm} imply  $f(x^{k+1}) - f(x^k) \leq \sigma\tau_k \big\langle \nabla f(x^{k}),  w^k-x^{k} \big\rangle+\nu_k$. The combination of the last two inequalities together with Lemma~\ref{Le:tauminArm} yields
	$$
		f(x^k) - f(x^{k+1})+\nu_k \geq \sigma\tau_k \frac{1}{2\alpha_{\max}\mu} \|w^k-x^k\|^2 \geq \sigma \tau_{\min} \frac{1}{2\alpha_{\max}\mu} \|w^k-x^k\|^2.
	$$
	Hence, performing the sum of the above inequality for $k= 0, 1,\ldots, N-1$, we conclude that
	$$
		\sum_{k= 0}^{N-1} \|w^k - x^k\|^2 \leq \frac{2{\alpha_{\max}}\mu \big[f(x^0) - f(x^{N+1})+ \sum_{k= 0}^{N}\nu_k\big]}{\sigma \tau_{\min}}\leq \frac{2{\alpha_{\max}}\mu \left[ f(x^0) - f^*+ \sum_{k= 0}^{\infty}\nu_k\right]}{\sigma \tau_{\min}},
	$$
	which implies the desired result.
\end{proof}

Next theorem presents an iteration-complexity bound for the sequence $\left(f(x^k)\right)_{k\in\mathbb{N}}$ when $f$ is convex.

\begin{theorem} \label{th:ccconv}
	Let $f$ be a convex function on $C$. Then, for every $N \in \mathbb{N}$, there holds
	$$
		\min \left\{f(x^k) - f^* :~k = 0, 1 \ldots, N-1\right\} \leq \frac{\|x^0 - x^*\|^2_{D_0} + \xi\left[f(x^0)-f^*+ \sum_{k=0}^{\infty} \nu_k\right]}{2 \alpha_{\min} \tau_{\min}}\frac{1}{N}.
	$$
\end{theorem}

\begin{proof}
	Using the first inequality in \eqref{eq:TolArm} and Lemma~\ref{Le:tauminArm}, we have $2 \alpha_{\min} \tau_{\min} \leq 2 \alpha_k \tau_k$, for all $k\in {\mathbb N}$. We also know form the convexity of $f$ that $\langle \nabla f(x^k), x^*-x^k \rangle \leq f^* - f(x^k)$, for all $k\in {\mathbb N}$. Thus, applying Lemma \ref{Le:xkArm} with $x=x^*$, after some algebraic manipulations, we conclude
	$$
		2 \alpha_{\min} \tau_{\min} \left[f(x^k)-f^*\right] \leq \|x^k-x^*\|_{D_k}^2-\|x^{k+1}-x^*\|_{D_{k+1}}^2 + \xi \left[f(x^k) - f(x^{k+1})+\nu_k \right] \quad k = 0, 1, \ldots.
	$$
	Hence, performing the sum of the above inequality for $k = 0,1,\ldots, N-1$, we obtain
	$$
		2 \alpha_{\min} \tau_{\min}\sum_{k=0}^{N-1} \left[f(x^k)-f^*\right] \leq \|x^0-x^*\|_{D_0}^2-\|x^{N+1}-x^*\|_{D_{N}}^2 + \xi\Big[f(x^0)-f(x^{N+1})+ \sum_{k=0}^{N-1} \nu_k\Big].
	$$
	Thus, $2\alpha_{\min} \tau_{\min} N \min\{f(x^k) - f^*:~ k = 0, 1 \ldots, N-1\} \leq \|x^0 - x^*\|^2_{D_0}+ \xi[f(x^0)-f^*+ \sum_{k=0}^{N-1} \nu_k]$, which implies the desired inequality.
\end{proof}

We ended this section with some results regarding the  number of function evaluations performed by Algorithm~\ref{Alg:GeneralSeach}.
Note that the computational cost associated to each (outer) iteration involves a gradient evaluation, the computation of a (inexact) projection, and evaluations of $f$ at different trial points.
Thus, we must consider the function evaluations at the rejected trial points.

\begin{lemma} \label{eq:nfeas}
	Let $N_{k}$ be  the number of function evaluations after $k\geq 0$ iterations of Algorithm~\ref{Alg:GeneralSeach}. Then,  $N_{k}\leq 1+ (k+1)[\log (\tau_{\min})/\log (\bar\omega)+1].$
\end{lemma}
\begin{proof}
	Let $j(k)\geq 0$ be the number of inner iterations in Step~2 of Algorithm~\ref{Alg:GeneralSeach} to compute the step size $\tau_k$.  Thus, $\tau_k\leq {\bar\omega}^{j(k)}$.
Using Lemma \ref {Le:tauminArm}, we have  $0< \tau_{\min}\leq  \tau _{k}$ for all  $k\in \mathbb{N}$, which implies that  $ \log \left( \tau_{\min} \right) \leq \log (\tau _{k})=j(k) \log (\bar\omega)$, for all $k \in \mathbb{N}$.  Hence, due to  $\log(\bar\omega) <0$, we have $ j(k) \leq \log (\tau_{\min})/\log (\bar\omega)$. Therefore,
	$$
	N_k =	1+ \sum _{\ell=0}^{k}(j(\ell)+1)\leq 1+  \sum _{i=0}^{k} \Big(\frac{\log (\tau_{\min})}{\log (\bar\omega)} +1\Big)= 1+(k+1) \Big(\frac{\log (\tau_{\min})}{\log (\bar\omega)}+1\Big),
	$$
	where the first equality follows from the definition of $N_k$.
\end{proof}

\begin{theorem}
	For a given $\epsilon>0$, Algorithm \ref{Alg:GeneralSeach} computes $x^k$ and $w^k$ such that $\|  w^{k}-x^{k}\|\leq \epsilon$ using,  at most,
$$1+\left({\frac{2{\alpha_{\max}}\mu\left[f(x^0)-f^* +\sum_{k= 0}^{\infty}\nu_k\right] }{\sigma \tau_{\min}}} \frac{1}{\epsilon^2}+1\right) \Big(\frac{\log (\tau_{\min})}{\log (\bar\omega)}+1\Big)$$
		function evaluations.
	
\end{theorem}
\begin{proof}
	The proof follows straightforwardly from Theorem~\ref{eq:theocomp} and Lemma~\ref{eq:nfeas}.
\end{proof}

\begin{theorem}
	Let $f$ be a convex function on $C$. For a given $\epsilon>0$, the number  of  function evaluations performed by  Algorithm \ref{Alg:GeneralSeach} is,  at most,
	$$1+\left(\frac{\|x^0 - x^*\|^2_{D_0} + \xi\left[f(x^0)-f^*+ \sum_{k=0}^{\infty} \nu_k\right]}{2 \alpha_{\min} \tau_{\min}}\frac{1}{\epsilon} + 1 \right)\Big(\frac{\log (\tau_{\min})}{\log (\bar\omega)}+1\Big),$$
	to compute $x^k$ such that $f(x^k) - f^*\leq \epsilon$.
\end{theorem}
\begin{proof}
	The proof follows straightforwardly from Theorem~\ref{th:ccconv} and  Lemma~\ref{eq:nfeas}.
\end{proof}


\section{Numerical experiments} \label{Sec:NumExp}

This section presents some numerical experiments in order to illustrate the potential advantages of considering inexact schemes in the SPG method.
We will discuss inexactness associated with both the projection onto the feasible set and the line search procedure. 

Given $A$ and $B$ two $m\times n$ matrices, with $m\geq n$, and $c\in\R$, we consider the matrix function $f:\R^{n\times n}\to\R$ given by:
\begin{equation}\label{objfun}
f(X):=\ds\frac{1}{2}\|AX-B\|^2_F + \sum_{i=1}^{n-1} \left[ c \left( X_{i+1,i+1}-X_{i,i}^2 \right)^2 + (1-X_{i,i})^2   \right],
\end{equation}
which combines a least squares term with a Rosenbrock-type function. 
Throughout this section, $X_{i,j}$ stands for the $ij$-element of the matrix $X$ and $\|\cdot\|_F$ denotes the Frobenius matrix norm, i.e., $\|A\|_F:=\sqrt{\langle A,A \rangle}$ where the inner product is given by $\langle A,B \rangle = \tr(A^TB)$.
The test problems consist of minimizing $f$ in \eqref{objfun} subject to two different feasible sets, as described below.
We point out that interesting applications in many areas emerge as constrained least squares matrix problems, see \cite{BirginMartinezRaydan2003} and references therein.
In turn, the Rosenbrock term was added in order to make the problems more challenging.

\begin{description}
\item[{\bf Problem~I:}]
\begin{equation*} \label{sdd}
	\begin{array}{cl}
		\ds\min     & f(X)           \\
		\mbox{s.t.} & X \in SDD^+,   \\
		            & L\leq X\leq U,
	\end{array}
\end{equation*}
where $SDD^+$ is the cone of symmetric and diagonally dominant real matrices with positive diagonal, i.e., 
$$SDD^+ :=\{X\in\R^{n\times n}\mid X=X^T, \; X_{i,i}\geq \ds\sum_{j\neq i}|X_{i,j}| \; \forall i\},$$
$L$ and $U$ are given $n\times n$ matrices, and  $L\leq X\leq U$ means that $L_{i,j} \leq X_{i,j} \leq U_{i,j}$ for all $i,j$.
The feasible set of Problem I was considered, for example, in the numerical tests of \cite{BirginMartinezRaydan2003}.

\item[{\bf Problem~II:}]
\begin{equation*} \label{spec}
	\begin{array}{cl}
		\ds\min     & f(X)                  \\
		\mbox{s.t.} & X \in \mathbb{S}^n_+, \\
		            & \tr(X)=1,
	\end{array}
\end{equation*}
where $\mathbb{S}^n_+$ is the cone of symmetric and positive semidefinite real matrices and $\tr(X)$ denotes the trace of $X$. 
The feasible set of Problem II was known as {\it spectrahedron} and appears in several interesting applications see, for example, \cite{allen2017linear,douglasprojected} and references therein.
\end{description}


It is easy to see that the feasible set o Problem~I is a closed and convex set and the feasible set of Problem~II is  a compact and convex set. 
As discussed in Section~\ref{sec:pcip}, the Dykstra's alternating algorithm and the Frank-Wolfe algorithm can be used to calculate inexact projections. The choice of the most appropriate method depends on the structure of the feasible set under consideration. 
For Problem~I, we used the Dykstra's algorithm described in \cite{BirginMartinezRaydan2003}, see also \cite{dykstraSDD}. In this case, $SDD^+=\ds\cap_{i=1}^n SDD^+_i$, where 
$$SDD^+_i :=\{X\in\R^{n\times n}\mid X=X^T, \; X_{i,i}\geq \ds\sum_{j\neq i}|X_{i,j}| \} \; \mbox{for all} \; i=1,\ldots,n,$$
and the projection of a given $Z\in\R^{n\times n}$ onto $SDD^+$ consists of cycles of projections onto the convex sets $SDD^+_i$.
Here an iteration of the Dykstra's algorithm should be understood as a complete cycle of projections onto all $SDD^+_i $ sets and onto the box $\{X\in\R^{n\times n}\mid L\leq X\leq U\}$.
Recall that this scheme provides an inexact projection as in Definition~\ref{def:InexactM}.
Now consider Problem~II. It is well known that calculating an exact projection onto the spectrahedron (i.e., onto the feasible set of Problem~II) requires a complete spectral decomposition, which can be prohibitive specially in the large scale case. In contrast, the computational cost of an iteration of the Frank-Wolfe algorithm described in Algorithm~\ref{Alg:CondG} is associated by an extreme eigenpair computation, see, for example, \cite{Jaggi2013}. Unfortunately, despite its low cost per-iteration, the Frank-Wolfe algorithm suffers from a slow convergence rate.
Thus, we considered a variant of the Frank-Wolfe algorithm proposed in \cite{allen2017linear}, which improves the convergence rate and the total time complexity of the classical Frank-Wolfe method. This algorithm specialized for the projection problem over the spectrahedron is carefully described in \cite{aguiar2021inexact}.
Without attempting to go into details, it replaces the top eigenpair computation in Frank-Wolfe with a top-$p$ (with $p\ll n$) eigenpair computation, where $p$ is an algorithmic parameter automatically selected.
The total number of computed eigenpairs can be used to measure the computational effort to calculate projections.
We recall that a Frank-Wolfe type scheme provides an inexact projection as in Definition~\ref{def:InexactProjC}.

We notice that Problems~I and II can be seen as particular instances of the problem~\eqref{eq:OptP} in which the number of variables is $(n^2+n)/2$. This mean that they can be solved by using Algorithm~\ref{Alg:GeneralSeach}.
We are especially interested in the spectral gradient version \cite{spgsiam} of the SPG method, which is often associated with large-scale problems \cite{JSSv060i03}. For this, we implemented Algorithm~\ref{Alg:GeneralSeach} considering $D_k:=I$ for all $k$, $\alpha_0 := \min(\alpha_{\max}, \max(\alpha_{\min}, 1/ \| \nabla f(x^0) \|))$ and, for $k>0$,
\begin{equation*}\label{spg}
\alpha_k:=\left\{\begin{array}{ll}
\ds\min (\alpha_{\max},\max (\alpha_{\min},\langle s^k,s^k\rangle/\langle s^k,y^k\rangle)), & \mbox{if} \; \langle s^k,y^k\rangle > 0\\
\alpha_{\max}, & \mbox{otherwise},
\end{array}\right.
\end{equation*}
where $s^k:=X^k - X^{k-1}$, $y^k:=\nabla f(X^k) - \nabla f(X^{k-1})$, $\alpha_{\min}=10^{-10}$, and $\alpha_{\max}=10^{10}$.
We set $\sigma = 10^{ -4}$, $\underline\tau=0.1$, $\bar\tau=0.9$, $\mu=1$ and $\nu_0=0$. Parameter $\delta_{\min}$ was chosen according to the line search used (see Section~\ref{Sec:SGM}), while parameter $\zeta_{\min}$ depends on the inexact projection scheme considered.

In the line search scheme (Step~2 of Algorithm~\ref{Alg:GeneralSeach}), if a step size $\tau_{\textrm{trial}}$  is not accepted, then $\tau_{\textrm{new}}$ is calculated using one-dimensional quadratic interpolation employing the safeguard $\tau_{\textrm{new}}\gets \tau_{\textrm{trial}}/2$  when the minimum of the one-dimensional quadratic lies outside $[\underline\omega \tau_{\textrm{trial}}, \bar\omega \tau_{\textrm{trial}} ]$, see, for example,  \cite[Section 3.5]{nocedal2006numerical}.
Concerning the stopping criterion, all runs were stopped at an iterate $X^k$ declaring convergence if
$$\ds \max_{i,j} (|X^k_{i,j}-W^k_{i,j}| )\leq 10^{-6},$$ 
where $W^k$ is as in \eqref{eq:PInexArm}.
Our codes are written in Matlab and are freely available at \url{https://github.com/maxlemes/InexProj-SGM}. All experiments were run on a macOS 10.15.7 with 3.7GHz Intel Core i5 processor and 8GB of RAM.

\subsection{Influence of the inexact projection} \label{sec:forcing}

We begin the numerical experiments by checking the influence of the forcing parameters that control the degree of inexactness of the projections in the performance of the SPG method.
In this first battery of tests, we used Armijo line searches, see Section~\ref{Sec:SGM}.

We generated 10 instances of Problem~I using $n=100$, $m=200$, and $c=10$. The matrices $A$ and $B$ were randomly generated with elements belonging to $[-1,1]$. We set $L\equiv 0$ and $U\equiv \infty$ as in \cite{BirginMartinezRaydan2003}. For each instance, the starting point $X^0$ was randomly generated with elements belonging to $[0,1]$, then it was redefined as $(X^0 + (X^0)^T )/2$ and its diagonal elements were again redefined as $2\sum_{j\neq i}^n X_{i,j}$, ensuring a feasible starting point. Figure~\ref{SDD} shows the average number of iterations,  the average number of Dykstra’s iterations, and the average  CPU time in seconds needed to reach the solution for different choices of $\zeta_k$, namely, $\zeta_k=0.99$, $0.9$, $0.8$, $0.7$, $0.6$, $0.5$, $0.4$, $0.3$, $0.2$, and $0.1$ for all $k$. 
Remember that {\it smaller values} of  $\zeta_k$ imply {\it more inexact} projections. As expected, the number of iterations of the SPG tended to increase as $\zeta_k$ decreased, see  Figure~\ref{SDD}(a).
On the other hand, the computational cost of an outer iteration (which can be measured by the number of Dykstra’s iterations) tends to decrease when considering smaller values of $\zeta_k$. 
This suggests a trade-off, controlled by parameter $\zeta_k$, between the number and the cost per iteration.
Figure~\ref{SDD}(b) shows that values for $\zeta_k$ close to 0.8 showed better results, which is in line with the experiments reported in \cite{BirginMartinezRaydan2003}.
Finally, as can be seen in Figure~\ref{SDD}(c), the CPU time was shown to be directly proportional to the number Dykstra’s iterations.


\begin{figure}[H]\centering
	\begin{tabular}{ccc}
		\includegraphics[scale=\myscale]{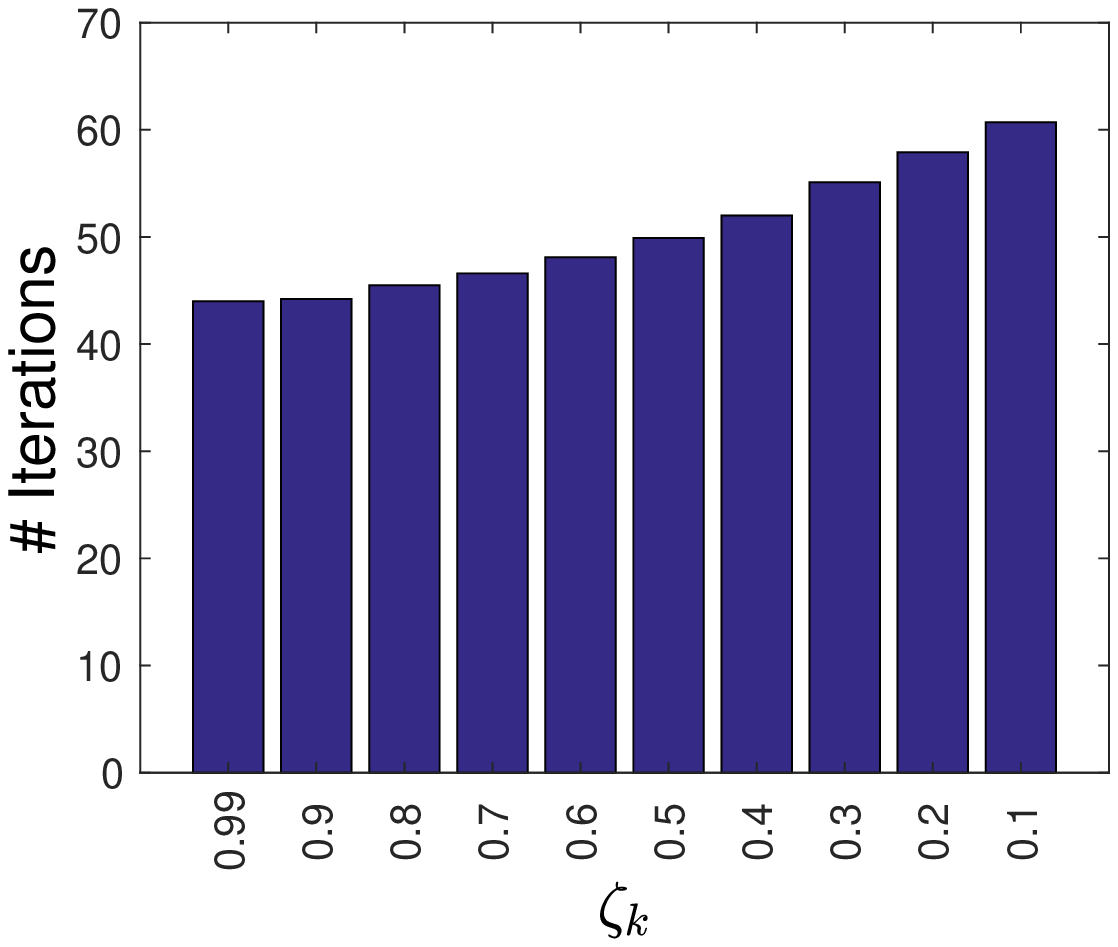} & \includegraphics[scale=\myscale]{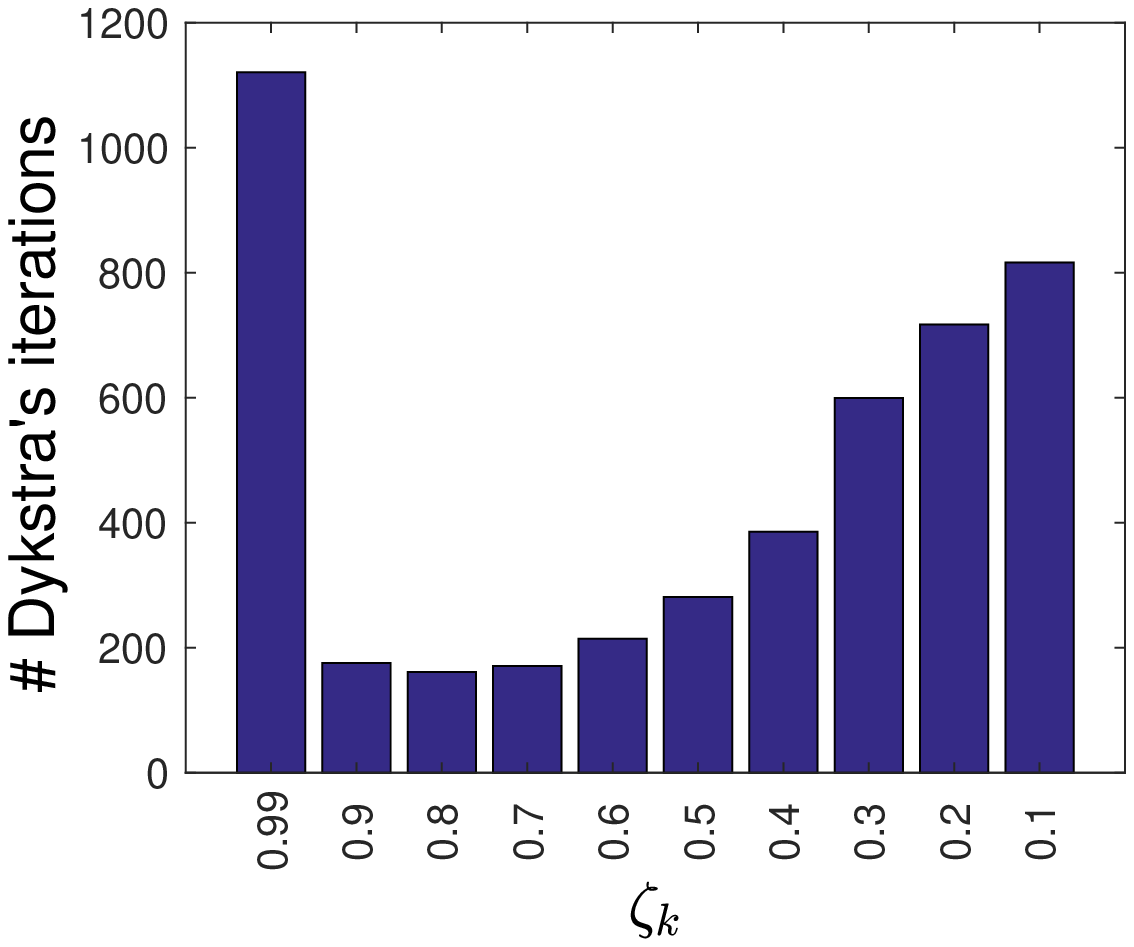} & \includegraphics[scale=\myscale]{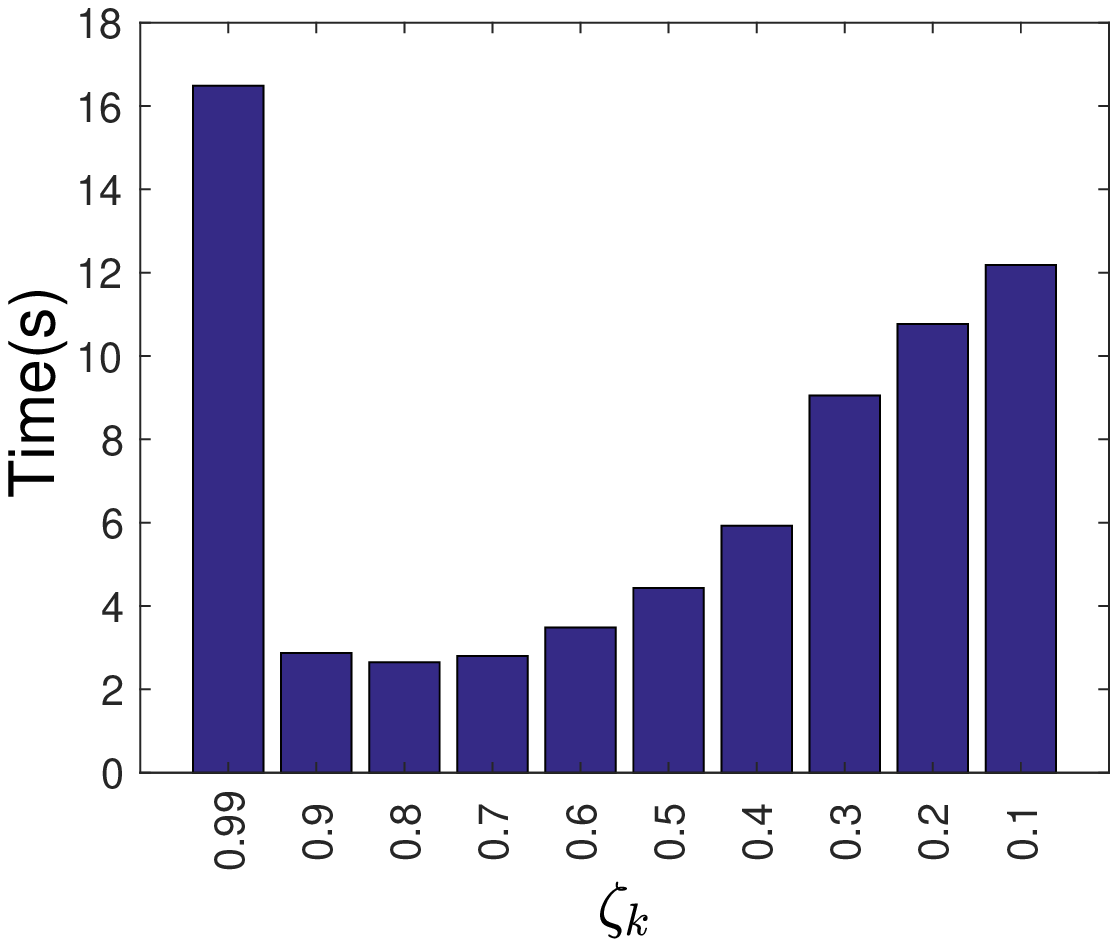}\\
		(a) & (b) & (c)\\
	\end{tabular}
	\caption{Results for 10 instances of Problem I using $n=100$, $m=200$, and $c=10$. Average number of: (a) iterations; (b) Dykstra’s iterations; (c)  CPU time in seconds needed to reach the solution for different choices of $\zeta_k$.}
	\label{SDD}
\end{figure}

Although Algorithm~\ref{Alg:GeneralSeach} is given only in terms of parameter $\zeta_k$, we will directly consider parameter $\gamma_k$ for Problem~II in which inexact projections are computed according to Definition~\ref{def:InexactProjC}. We randomly generated 10 instances of Problem~II with $n=800$, $m=1000$, and $c=100$. Matrices $A$ and $B$ were obtained similarly to Problem~I. In turn, a starting point $X^0$ was randomly generated with elements in the interval $[-1,1]$, then it was redefined to be $X^0  (X^0)^T/\tr(X^0  (X^0)^T)$, resulting in a feasible initial guess.
Figure~\ref{Spec} shows the average number of iterations,  the average number of computed eigenpairs, and the average  CPU time in seconds needed to reach the solution for different constant choices of $\gamma_k$ ranging from $10^{-8}$ to $0.4999$. Now, {\it higher values} of $\gamma_k$ imply {\it more inexact} projections. Note that for appropriate choices of $\zeta_k$, the adopted values of $\gamma_k$ fulfill Assumption A1 of Section~\ref{Sec:FullConvRes}. Concerning the number of iterations, as can be seen in Figure~\ref{Spec}(a), the SPG method was not very sensitive to the choice of parameter $\gamma_k$. Hence, since higher values of $\gamma_k$ imply cheaper iterations, the number of computed eigenpairs and the CPU time showed to be inversely proportional to $\gamma_k$, see Figures~\ref{Spec}(b)--(c). Thus, our experiments suggest that the best value for $\gamma_k$ seems to be $0.4999$.

\begin{figure}[H]\centering
	\begin{tabular}{ccc}
		\includegraphics[scale=\myscale]{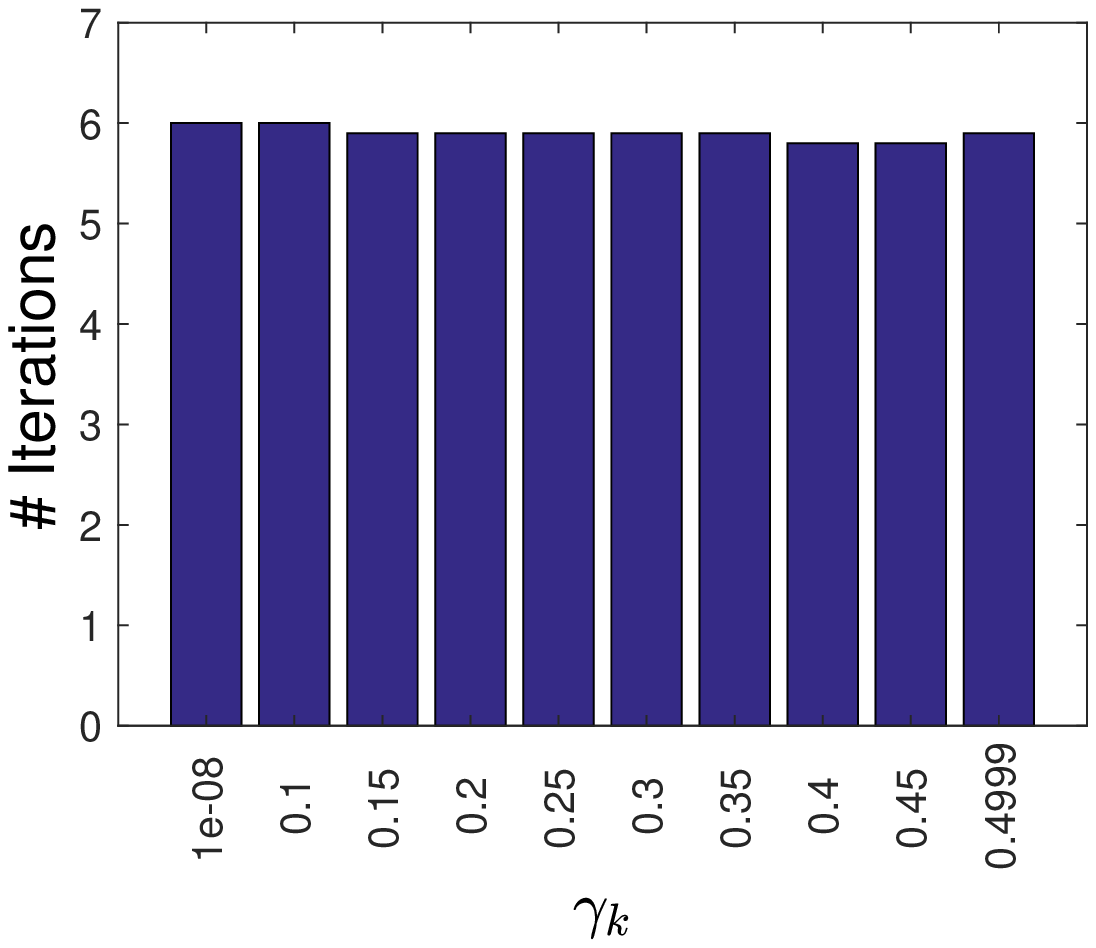} & \includegraphics[scale=\myscale]{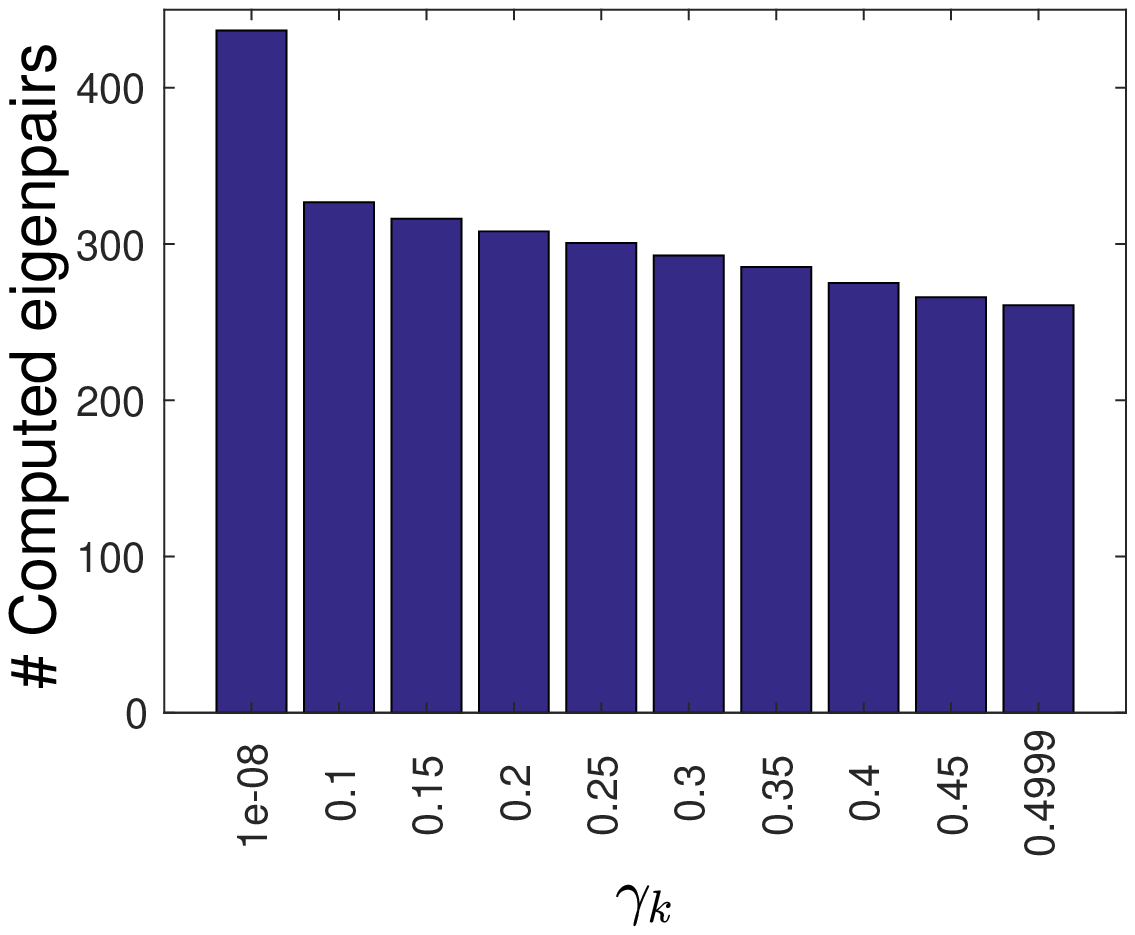} & \includegraphics[scale=\myscale]{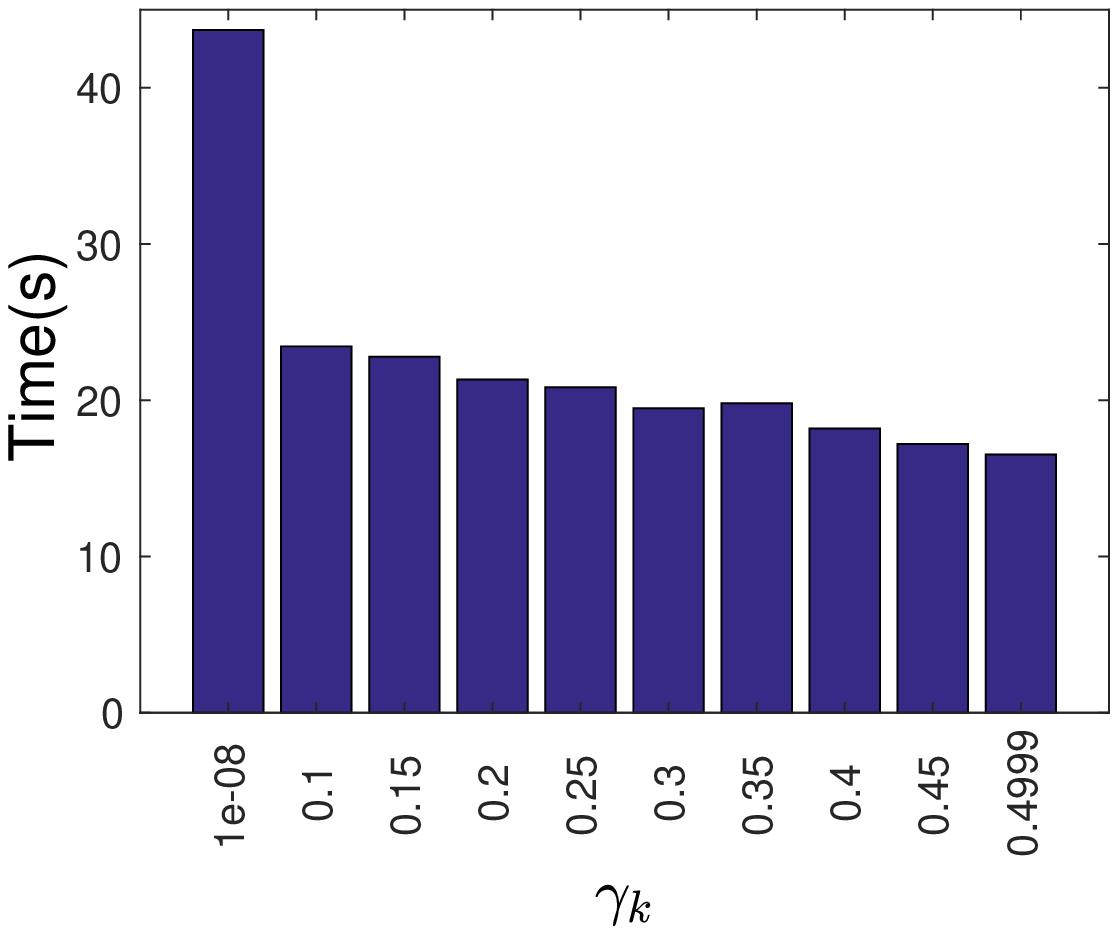} \\
		(a) & (b) & (c)\\
	\end{tabular}
	\caption{Results for 10 instances of Problem II using $n=800$, $m=1000$, and $c=100$. Average number of: (a) iterations; (b) computed eigenpairs; (c)  CPU time in seconds needed to reach the solution for different choices of $\gamma_k$.}
	\label{Spec}
\end{figure}

\subsection{Influence of the line search scheme}

The following experiments compare the performance of the SPG method with different strategies for computing the step sizes. We considered the Armijo, the Average-type, and the Max-type line searches discussed in Section~\ref{Sec:SGM}. 
Based on our numerical experience, we employed the fixed value $\eta_k=0.85$ for the Average-type line search and $M=5$ for the Max-type line search. 
According to the results of the previous section, we used the fixed forcing parameters $\zeta_k=0.8$ and $\gamma_k=0.4999$ to compute inexact projections for Problems~I and II, respectively.

We randomly generated 100 instances of each problem as described in Section~\ref{sec:forcing}. The dimension of the problems and the parameter $c$ in \eqref{objfun} were also taken arbitrarily. For Problem~I, we choose $100\leq n \leq 800$ and $10\leq c \leq 50$, whereas for Problem~II, we choose $10\leq n \leq 200$ and $100\leq c \leq 1000$. In both cases, we set $m=2n$.
We compare the strategies with respect to the number of function evaluations, the number of (outer) iterations, the total computational effort to calculate projections (measured by the number of Dykstra’s iterations and computed eigenpairs for Problems~I and II, respectively), and the CPU time. The results are shown in Figures~\ref{ppSDD} and \ref{ppSpec} for Problems~I and II, respectively, using performance profiles \cite{dolan2002benchmarking}.

For Problem~I, with regard to the number of function evaluations, the SPG method with the Average-type line search was the most efficient among the tested strategies.
In a somewhat surprising way, in this set of test problems, the Armijo strategy was better than the Max-type line search, see Figure~\ref{ppSDD}(a).
On the other hand, as can be seen in Figure~\ref{ppSDD}(b), the Armijo strategy required fewer iterations than the non-monotonous strategies.
As expected, this was reflected in the number of Dykstra’s iterations and the CPU time, see Figures~\ref{ppSDD}(c)--(d).
We can conclude that, with respect to the last two criteria, the Armijo and Average-type strategies had similar and superior performances to the Max-type strategy.

Now, concerning Problem~II, Figure~\ref{ppSpec} shows that the non-monotonous strategies outperformed the Armijo strategy in all the comparative criteria considered.
Again,  the Average-type strategy seems to be superior to the Max-type strategy.

\begin{figure}[H]\centering
	\begin{tabular}{cccc}
		\includegraphics[scale=\myscale]{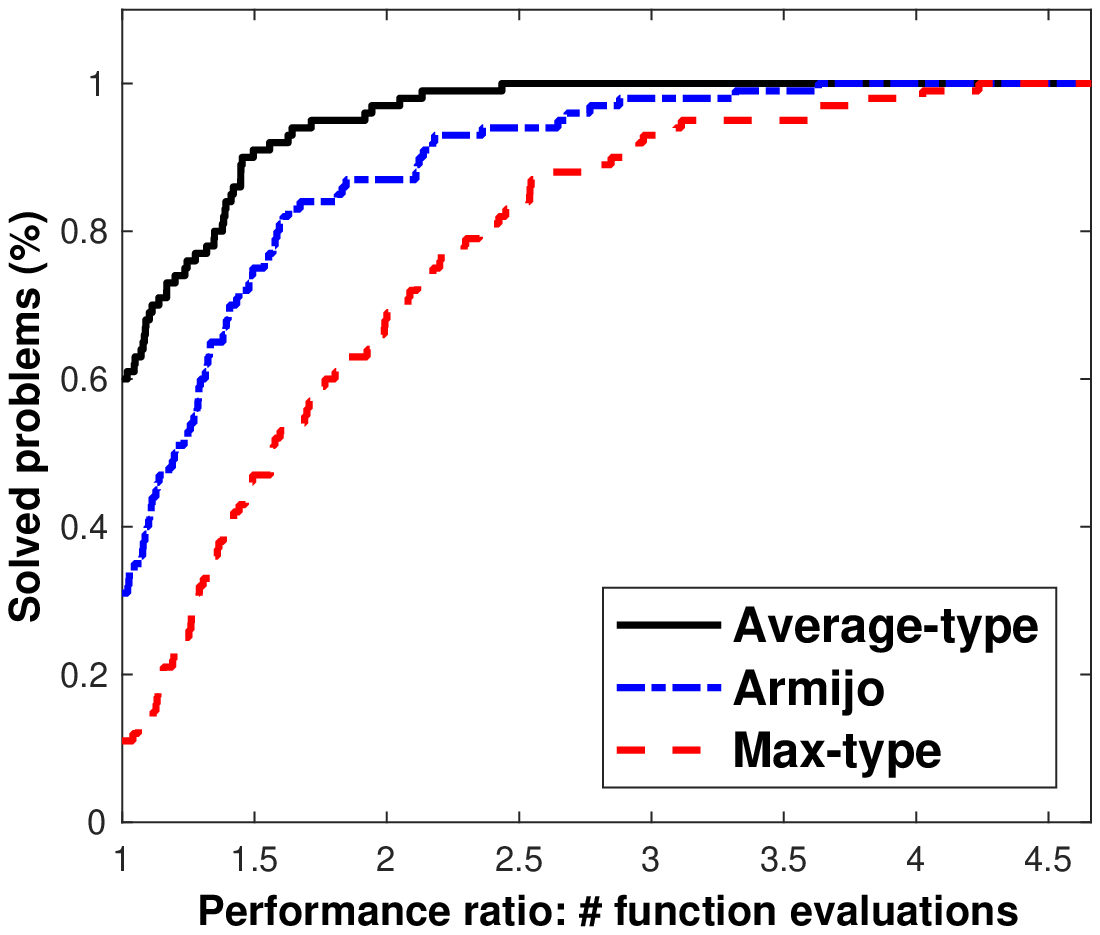} &\includegraphics[scale=\myscale]{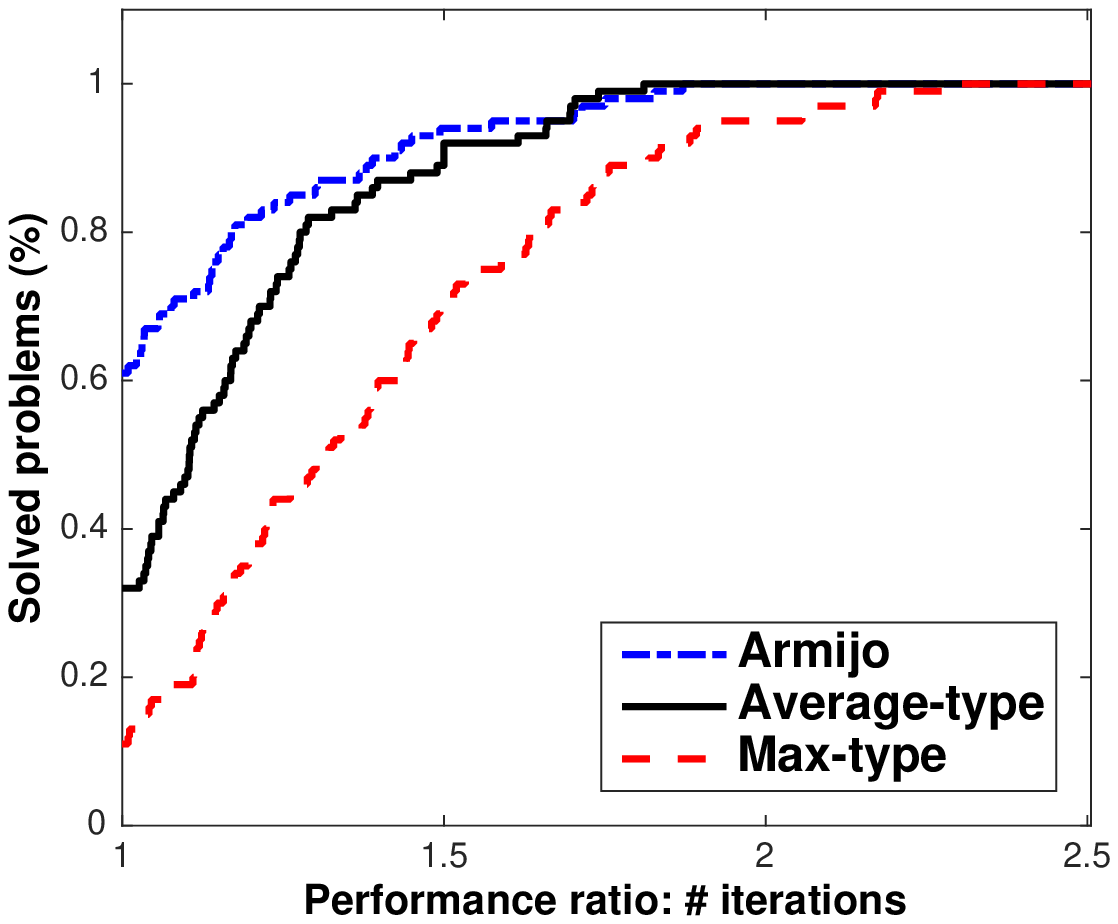} & \includegraphics[scale=\myscale]{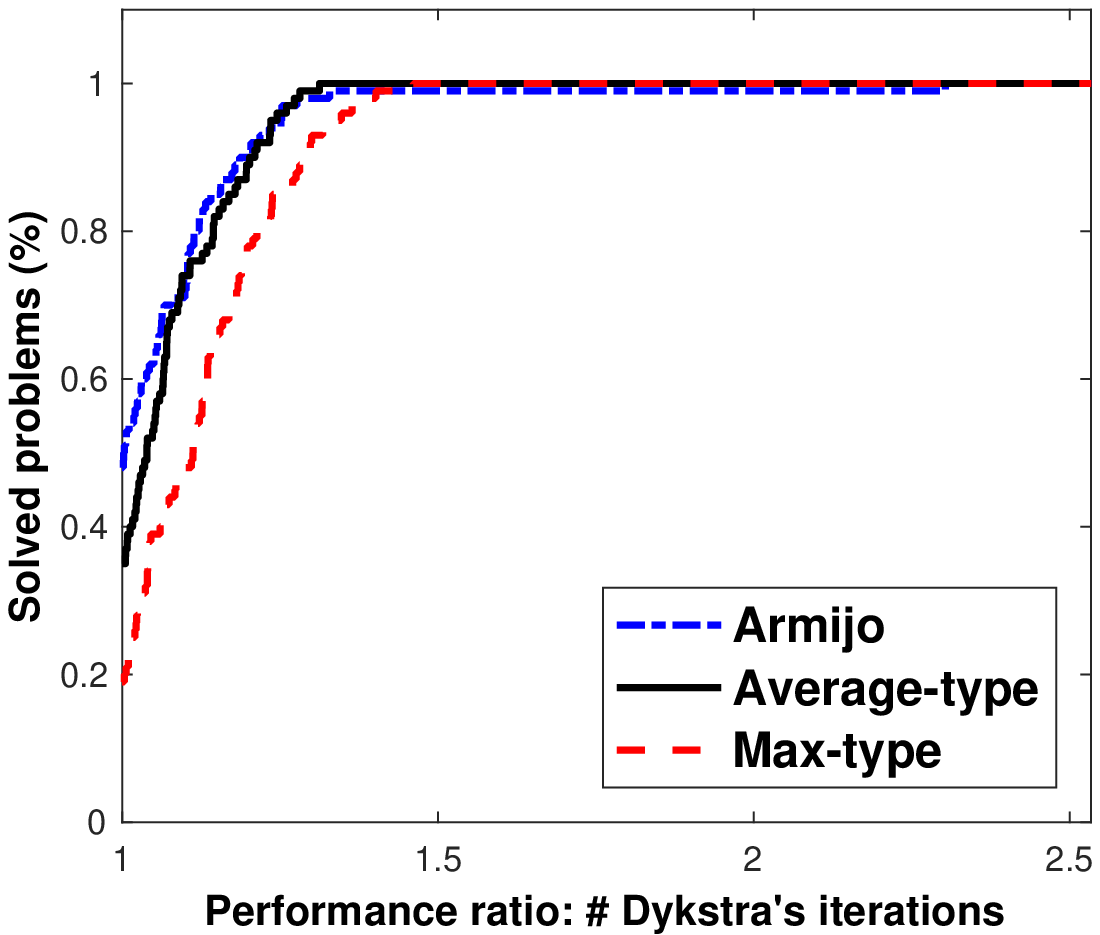} & \includegraphics[scale=\myscale]{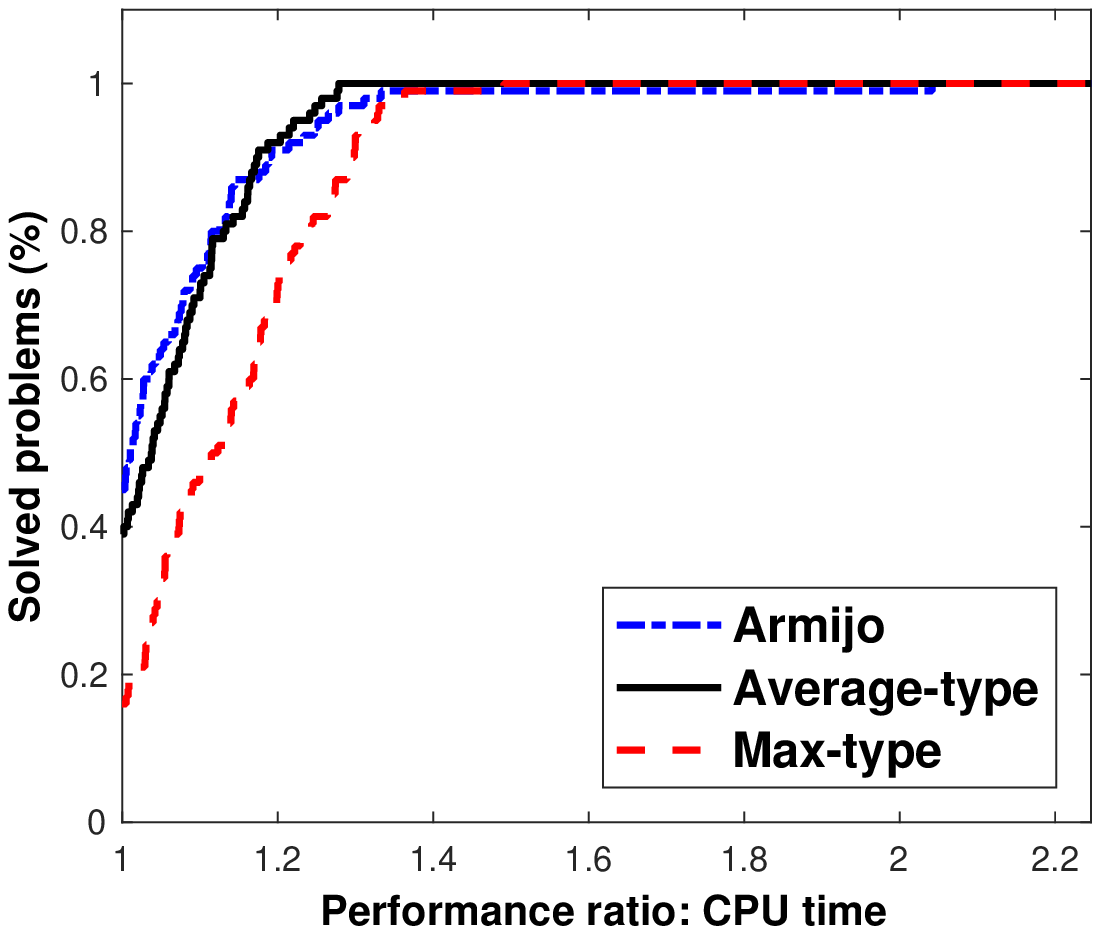} \\
		(a) & (b) & (c) & (d)\\
	\end{tabular}
	\caption{Performance profiles for Problem~I considering the SPG method with the Armijo, the Average-type, and the Max-type line searches strategies using as performance measurement: (a) number of function evaluations; (b) number of (outer) iterations; (c) number of Dykstra’s iterations; (d) CPU time.}
	\label{ppSDD}
\end{figure}

\begin{figure}[H]\centering
	\begin{tabular}{cccc}
		\includegraphics[scale=\myscale]{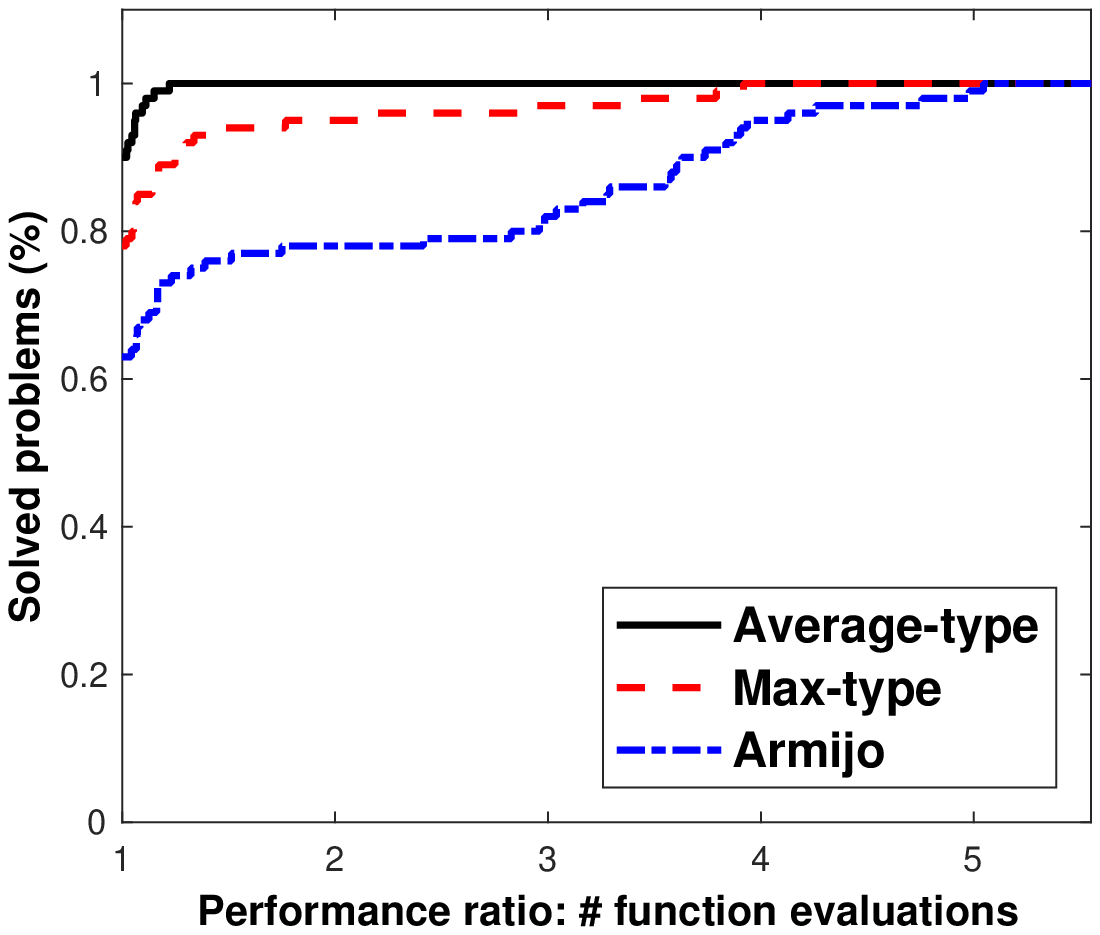}&\includegraphics[scale=\myscale]{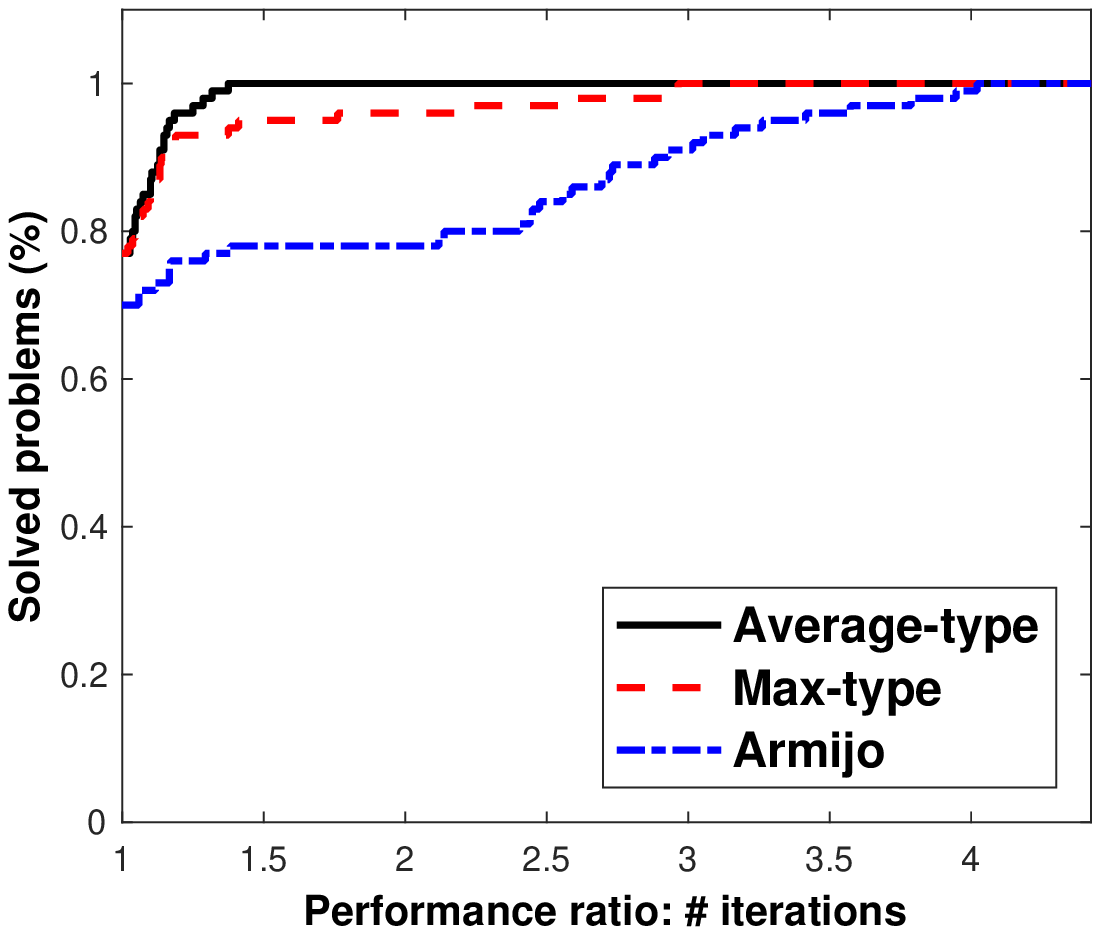} & \includegraphics[scale=\myscale]{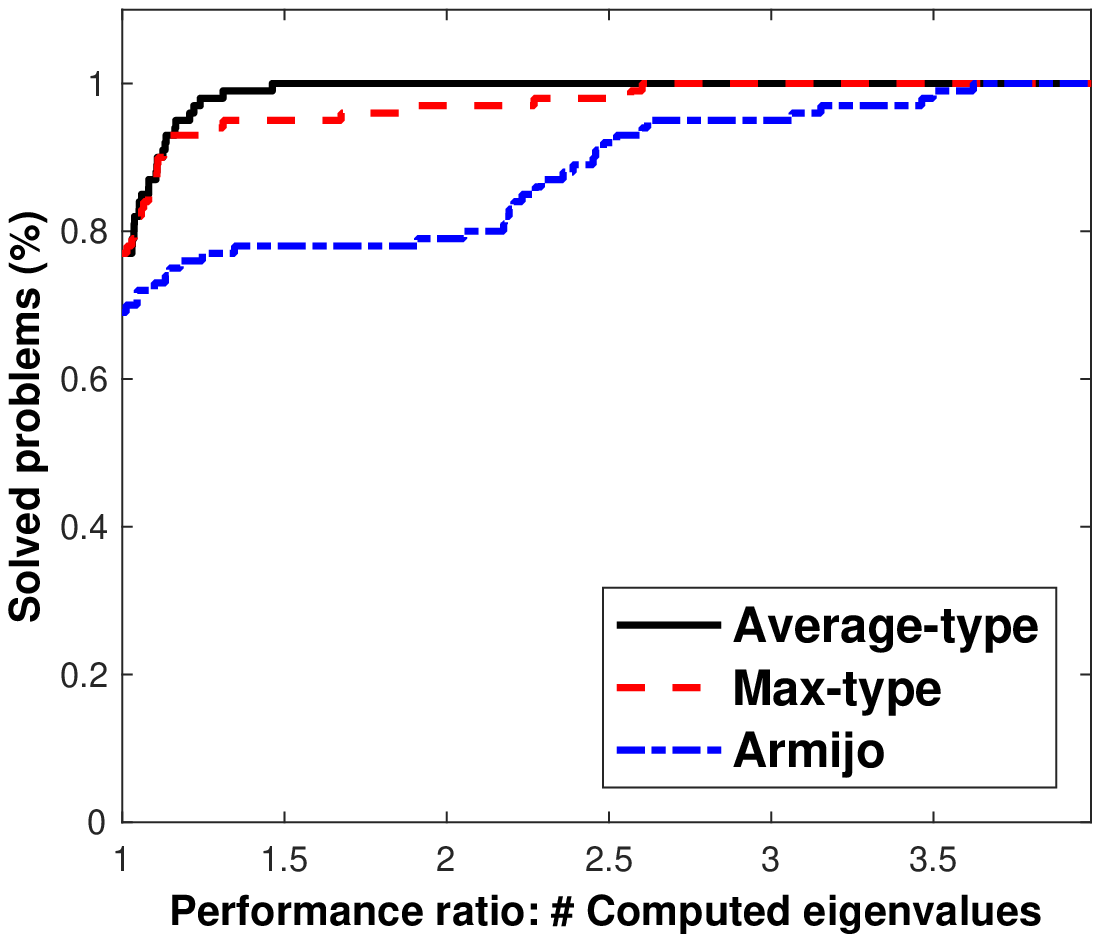} & \includegraphics[scale=\myscale]{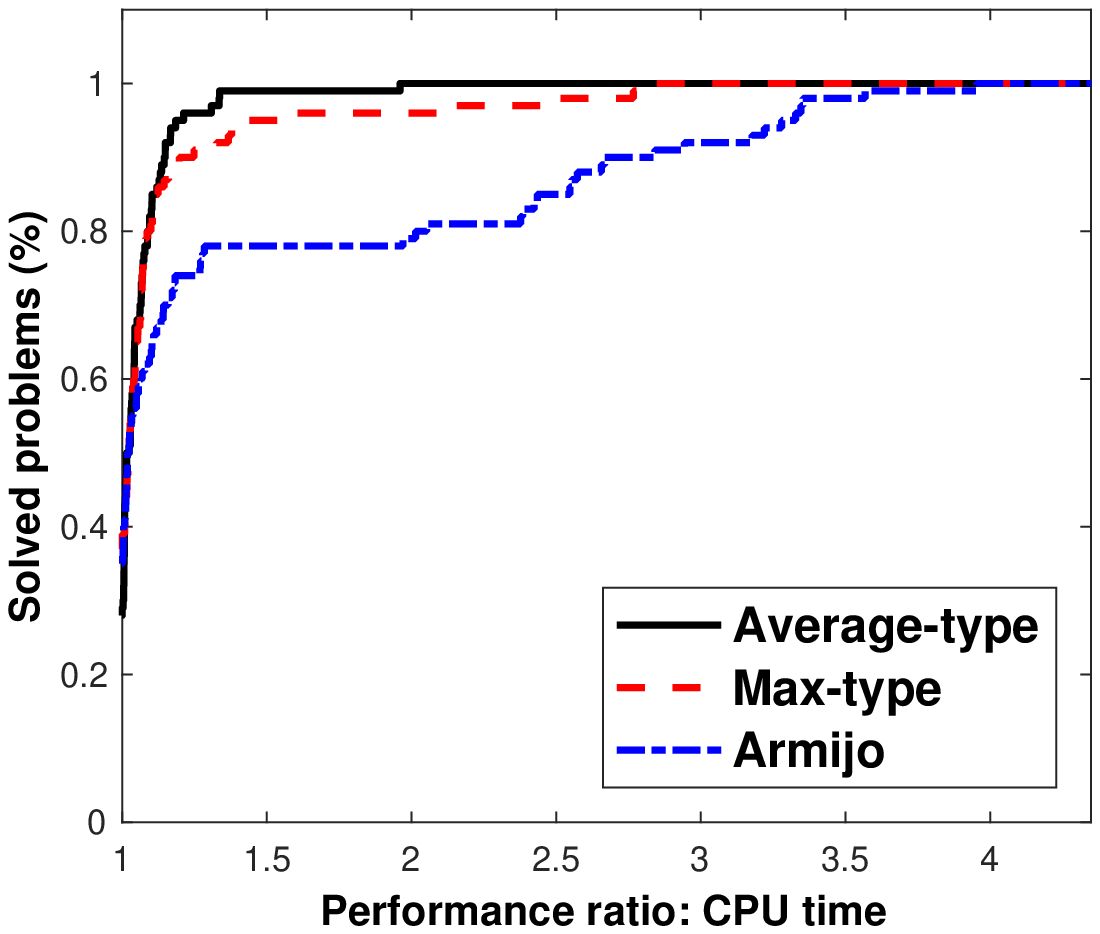} \\
		(a) & (b) & (c) & (d)\\
	\end{tabular}
	\caption{Performance profiles for Problem~II considering the SPG method with the Armijo, the Average-type, and the Max-type line searches strategies using as performance measurement: (a) number of function evaluations; (b) number of (outer) iterations; (c) number of computed eigenpairs; (d) CPU time.}
	\label{ppSpec}
\end{figure}

From all the above experiments, we conclude that the non-monotone line searches tend to require fewer objective function evaluations.
However, this does not necessarily mean computational savings, since there may be an increase in the number of iterations.
In this case, optimal efficiency of the algorithm comes from a compromise between those two conflicting tendencies.
Overall, the use of non-monotonous line search techniques is mainly justified when the computational effort of an iteration is associated with the cost of evaluating the objective function.

\section{Conclusions} \label{Sec:Conclusions}
In this paper, we study the SGP method  to solve   constrained convex optimization problems employing  inexact projections onto the feasible set and a general non-monotone  line search. We expect that this paper will contribute to the development of research in this field, mainly to solve large-scale problems when the computational effort of an iteration is associated with the projections onto the feasible set and  the cost of evaluating the objective function. Indeed, the idea of using the inexactness in the projection as well as in the line search,   instead of the exact ones, is particularly interesting from a computational point of view. In particular,   it is noteworthy that the Frank-Wolfe method  has a low computational cost per iteration  resulting in high computational performance in different classes of compact sets, see \cite{GarberHazan2015, Jaggi2013}.  An issue that deserves attention is the search for new efficient methods such as the Frank-Wolfe's and Dykstra's  methods that generate inexact projections.


\end{document}